\documentclass{article}

\usepackage{PRIMEarxiv}

\usepackage[utf8]{inputenc} % allow utf-8 input
\usepackage[T1]{fontenc}    % use 8-bit T1 fonts
\usepackage{hyperref}       % hyperlinks
\usepackage{url}            % simple URL typesetting
\usepackage{booktabs}       % professional-quality tables
\usepackage{amsfonts}       % blackboard math symbols
\usepackage{nicefrac}       % compact symbols for 1/2, etc.
\usepackage{microtype}      % microtypography
\usepackage{lipsum}
\usepackage{fancyhdr}       % header
\usepackage{graphicx}       % graphics
\graphicspath{{media/}}     % organize your images and other figures under media/ folder
\usepackage{amssymb,mathrsfs,dsfont, mathtools,amsthm}
\usepackage{bbm}
\usepackage{bm,url,wasysym}
\usepackage{amsmath}
\usepackage{mleftright}
\usepackage[round]{natbib}
\usepackage{xcolor}
\usepackage{multirow}

\theoremstyle{plain}
\newtheorem{theorem}{Theorem}[section]
\newtheorem{lemma}[theorem]{Lemma}
\newtheorem{corollary}[theorem]{Corollary}

\theoremstyle{remark}
\newtheorem{Def}{Definition}[section]
\newtheorem{assumption}{Assumption}[section]
\newtheorem{remark}{Remark}[section]

\newcommand{\pr}[1]{\left(#1\right)}
\newcommand{\sqbr}[1]{\left[#1\right]}
\newcommand{\br}[1]{\left\{#1\right\}}
\newcommand{\verts}[1]{\left\lvert#1\right\rvert}
\newcommand{\Verts}[1]{\left\lVert#1\right\rVert}
\newcommand{\KL}[2]{\operatorname{KL}\pr{#1 \parallel #2}}

\DeclareMathOperator*{\argmax}{arg\,max}
\DeclareMathOperator*{\argmin}{arg\,min}

\newcommand{\Nb}{\mathbb{N}}
\newcommand{\Pb}{\mathbb{P}}

\def\ba#1\ea{\begin{align*}#1\end{align*}} %\ba = \begin{algin*}, \ea = \end{align*}
\def\banum#1\eanum{\begin{align}#1\end{align}} %\banum = \begin{algin}, \eanum

\title{Nonparametric logistic regression with deep learning
}

\author{
  Atsutomo Yara \\
  Graduate School of Engineering Science, \\
  Osaka University \\
  \texttt{a-yara@sigmath.es.osaka-u.ac.jp} \\
   \And
  Yoshikazu Terada \\
  Graduate School of Engineering Science, \\
  Osaka University \\
  RIKEN Center for Advanced Intelligence Project (AIP) \\
  \texttt{terada.yoshikazu.es@osaka-u.ac.jp} \\
}

\begin{document}
\maketitle

\begin{abstract}
Consider the nonparametric logistic regression problem. 
In the logistic regression, we usually consider the maximum likelihood estimator, and the excess risk is the expectation of the Kullback-Leibler (KL) divergence between the true and estimated conditional class probabilities.
However, in the nonparametric logistic regression, 
the KL divergence could diverge easily, and thus, the convergence of the excess risk is difficult to prove or does not hold.
Several existing studies show the convergence of the KL divergence under strong assumptions. In most cases, our goal is to estimate the true conditional class probabilities. Thus, instead of analyzing the excess risk itself, 
it suffices to show the consistency of the maximum likelihood estimator 
in some suitable metric.
In this paper, using a simple unified approach for analyzing the nonparametric maximum likelihood estimator (NPMLE), 
we directly derive convergence rates of the NPMLE in the Hellinger distance under mild assumptions. 
Although our results are similar to the results in some existing studies, 
we provide simple and more direct proofs for these results. 
As an important application, we derive convergence rates of the NPMLE with fully connected deep neural networks and show that the derived rate nearly achieves the minimax optimal rate.
\end{abstract}


\section{Introduction} \label{sec:intro}
Deep neural networks (DNNs) have attracted attention in recent years for their extremely high performance in various complex real-world tasks such as image recognition and object detection.
Moreover, DNN-based generative models have achieved impressive results for more complicated tasks such as natural language processing and image generation.
Although there are still many unexplored theoretical aspects, some recent studies show the theoretical advantages of DNN-based estimators in several aspects.
The theoretical study of nonparametric estimation, such as nonparametric regression and nonparametric classification using DNNs, is one of the areas that have drawn considerable notice.

In nonparametric regression problems, 
theoretical analysis of deep learning has recently been studied in various settings. 
Many studies theoretically have discovered the advantages of deep learning and analyzed why deep neural networks outperform other nonparametric estimators, such as kernel methods
(e.g., \cite{ImaizumiFukumizu19, Schmidt-Hieber20, suzuki19,  Hayakawa&Suzuki2020, EckleSchmidt-Hideber19}).
Moreover, some of these studies theoretically indicate the possibility that the curse of dimensionality can be avoided by using deep neural networks (e.g., \cite{suzuki19, Schmidt-Hieber20, Nakada&Imaizumi2020, Suzuki&Nitanda2021, BauerKohler19, Schmidt-Hieber19, ChenEtAl22}).

On the other hand, for the classification problems, most studies focus on the convergence of the excess risk for the 0-1 loss (i.e., the classification error).
However, since minimizing the empirical 0-1 loss is numerically challenging and impractical, the estimator is obtained by minimizing an empirical risk for some surrogate loss, such as the hinge loss and the logistic loss. Thus, the excess risk for the 0-1 loss is evaluated indirectly by evaluating the excess risk for such losses.
In recent years, convergence rates of classification errors for the classification problem with DNNs have been studied in various settings (e.g., \cite{HuEtAl20, KimEtAl21, Kohler2020, Kohler2022, FengEtAl21, ZhouHuo24, KoEtAl23}).
\cite{HuEtAl20} show that in binary classification, the DNN classifier that minimizes the hinge loss achieves the minimax optimal convergence rates of the classification error under some assumptions.
They assume that the difference between the conditional density given a positive label and the conditional density given a negative label can be represented by a deep neural network model.
It is also assumed that the probability of the input data being near the decision boundary is small.
\cite{KimEtAl21} show that under the Tsybakov noise condition (\citealp{Tsybakov04}), the DNN classifier that minimizes the hinge loss achieves the minimax optimal convergence rates up to logarithmic factors for the classification error in three cases: when the true model has (1) smooth decision bounds, has (2) smooth conditional class probabilities, and satisfies (3) margin conditions (\citealp{SteinwartChristmannl08}).
\cite{ZhouHuo24} consider the binary classification of unbounded data generated by Gaussian mixture models.
They derive convergence rates of the classification error for the DNN classifier that minimizes the hinge loss.
Since their convergence rates do not depend on the dimension of the input, 
they demonstrate that deep learning can overcome the curse of dimensionality in their settings.
\cite{KoEtAl23} also consider the binary classification problem.
They assume that (1) the Tsybakov noise condition and (2) the true conditional class probability locally belongs to the Barron approximation space.
Then, they derive convergence rates of the classification error for the DNN classifier that minimizes the logistic loss.
Furthermore, they show that the derived convergence rate achieves the minimax optimal rate up to a logarithmic factor.
In \cite{Kohler2020}, \cite{Kohler2022}, and \cite{FengEtAl21}, convergence rates for the classifiers using convolutional neural networks (CNNs) are studied.
\cite{Kohler2022} assume smoothness conditions and a max-pooling structure for the true conditional class probabilities and derive convergence rates of the classification error for the CNN classifier that minimizes the squared loss, while \cite{Kohler2020} derive convergence rates of the classification error for the CNN classifier that minimizes the logistic loss under similar conditions.
\cite{FengEtAl21} analyze the approximation ability of the specific CNN with respect to the \(L^p\) norm to approximate functions from the Sobolev space.
Then, they derive convergence rates of the classification error for the CNN classifier that minimizes the loss \(\phi(v) = ((1 - v) \lor 0)^p\).
They also drive convergence rates of the classification error under the Tsybakov noise condition.

The convergence of estimated conditional class probabilities to the true conditional class probabilities is also an important topic in the classification problem.
Usually, DNN classifiers output the conditional probability of belonging to each class by using a logistic sigmoid function or a softmax function in the output layer.
The estimation of conditional class probabilities is important in practice as it provides more information than simply predicting class memberships.
For instance, in artificial intelligence systems for disease diagnosis, it is better to show the confidence level of the candidate's diagnoses obtained from the system.
Moreover, even in modern causal inference, the nonparametric estimation of the propensity score (the conditional probability of treatment) plays an important role. 
For example, in the semiparametric problem,
the double machine learning approach (\citealp{ChernozhukovEtAl18}) requires a convergence rate \(o_p(n^{-1/4})\) for the nonparametric estimation of the nuisance parameters.
If the nuisance parameters are conditional class probabilities, 
we need to check such a requirement on the convergence rates of the estimators. 

In this paper, we consider the nonparametric estimation of conditional class probabilities 
in logistic regression.
More precisely, we analyze the theoretical properties of the nonparametric maximum likelihood estimator (NPMLE) for the conditional class probabilities.
To show the consistency of NPMLEs, we may simply consider the convergence of the excess risk for the logistic loss. 
In logistic regression, the excess risk is the expectation of the Kullback-Leibler (KL) divergence between the true and estimated conditional class probabilities.
Unfortunately, it is known that the KL divergence between the true and estimated densities could easily diverge in the nonparametric estimation problem, and thus, strong assumptions are required to derive the convergence of the excess risk in logistic regression.
In fact, \cite{OhnKim22} study the convergence of the KL divergence for DNN estimators obtained by empirical risk minimization with regularization in the binary classification problem.
It is assumed that the true conditional class probability is bounded away from zero and one to prove the convergence of the excess risk.
To overcome this difficulty,
\cite{BosSchmidt-Hieber22} evaluates the truncated KL divergence between the true and estimated conditional class probabilities instead of the KL divergence between them for addressing the divergence issue of the excess risk.
They derive convergence rates of the expectation of the truncated KL divergence between the true and estimated conditional class probabilities by DNNs in the multi-class classification problem.
On the other hand,
\cite{Bilodeau2023} show convergence rates for the KL divergence between the smoothed maximum likelihood estimator of the conditional class probabilities and the true conditional probabilities when the true conditional class probabilities are included in the estimation model.

In this study, by using the technique developed in \cite{vandeGeer00}, 
we directly evaluate the squared Hellinger distance between the true and estimated conditional class probabilities and 
derive an oracle inequality for the nonparametric maximum likelihood estimator in multi-class logistic regression.
In our approach, since we do not consider the convergence of the KL divergence, the assumption that the true conditional class probability is bounded away from zero is not necessary.
As an important application of this oracle inequality, 
we derive convergence rates of the DNN estimators for the true conditional class probabilities 
when the true conditional class probabilities are the composition structured functions (see Section \ref{sec: convergence rates} for more detail). 
In the analysis of convergence rates for deep learning, the sparsity of DNN models is often imposed. That is, the number of non-zero parameters in the DNN is assumed to be much fewer than the total number of parameters (e.g., \cite{suzuki19}, \cite{Schmidt-Hieber20}, and \cite{BosSchmidt-Hieber22}). 
However, sparse feedforward neural networks are not widely used in practice. In this study, we derived convergence rates without imposing sparsity on the DNN model.
Furthermore, we show that the derived convergence rates of the DNN estimators are almost minimax optimal. 

Table \ref{tab: related works} shows a comparison between some existing studies and this study.
In this table, $\beta$ represents the smoothness of the true conditional class probability, $d$ denotes the input dimension, and $\alpha$ is the SVB index (see Definition \ref{def: small value bound}).
Furthermore, \(\beta_i^*\) and \(t_i^*\) are the smoothness and input dimension of each component of the composite functions, respectively (see Section \ref{sec: convergence rates} for details).
The symbol $\tilde{\mathcal{O}}$ represents the big O notation that ignores polylogarithmic factors.
The "Assumptions" column lists the assumptions imposed on the true conditional class probability.
The contributions of this study can be summarized in the following three points:
\begin{table}
    \begin{center}
        \caption{
        Comparison with related works.
        }
         \resizebox{1.00\textwidth}{!}{
         \begin{tabular}{cccccc}
             \toprule
             & Metric & Model & Function class & Convergence rate & Assumptions \\
             \midrule
             \multirow{2}{*}{\cite{OhnKim22}} & \multirow{2}{*}{KL divergence} & \multirow{2}{*}{Sparse DNN} & \multirow{2}{*}{Hölder} & \multirow{2}{*}{$\tilde{\mathcal{O}}\Bigl(n^{- \frac{2\beta}{2\beta + d}}\Bigr)$} & Bounded away \\
             & & & & & from zero \\
             \midrule
             \multirow{2}{*}{\cite{BosSchmidt-Hieber22}} & Truncated & \multirow{2}{*}{Sparse DNN} & \multirow{2}{*}{Hölder} & \multirow{2}{*}{$\tilde{\mathcal{O}}\Bigl(n^{- \frac{(1 + \alpha)\beta}{(1 + \alpha)\beta + d}}\Bigr)$} & \multirow{2}{*}{$\alpha$-SVB} \\
             & KL divergence & & & \\
             \midrule
             {\bf This work} & \bf Hellinger  & \bf Non-sparse & \bf Composition & \bf \multirow{2}{*}{$\tilde{\mathcal{O}}\Bigl(\max_i n^{- \frac{\beta_i^\ast}{ \beta_i^\ast + t_i^\ast}}\Bigr)$} & \bf \multirow{2}{*}{Assumption \ref{assumption}} \\
             (The minimax optimality is shown.)& \bf distance & \bf DNN & \bf structured & \bf  \\
             \bottomrule
        \end{tabular}
        \label{tab: related works}
        }
    \end{center}
\end{table}
\begin{itemize}
    \item 
    Under the realistic assumptions, we derive convergence rates of the nonparametric maximum likelihood estimator of conditional class probabilities by utilizing the general framework of \cite{vandeGeer00}.
    Compared with the existing works, our proofs are more straightforward and clear. 
    
    \item We show that when the true conditional class probability has a composite function structure, the convergence rate of nonparametric logistic regression using deep learning nearly achieves the minimax optimal convergence rate. 
    This minimax optimality is a nontrivial extension of Theorem 3 in \cite{Schmidt-Hieber20} not shown in \cite{BosSchmidt-Hieber22}.
    
    \item In contrast to the existing works, 
    we consider dense neural networks as models instead of sparse neural networks.
    Even in such a setting, we show that
    the NPMLEs can achieve the minimax optimal convergence rates up to logarithmic factors.
\end{itemize}

\section{Preliminaries}
\subsection{Notation}
In this section, we will organize notations.
The notations used in this paper are referenced from \cite{Schmidt-Hieber20} and \cite{BosSchmidt-Hieber22}.

Vectors and vector-valued functions are denoted in bold symbols.
The $k$-th component of a vector $\bm{v} \in \mathbb{R}^d$ is represented as $v_k$.
For \(0 < p < \infty\) and \(\bm{v} \in \mathbb{R}^d\), \(\lvert \bm{v} \rvert_p \coloneqq (\sum_{i=1}^d v_i^p)^{1/p}\).
For \(p = \infty\), we define \(\lvert \bm{v} \rvert_\infty \coloneqq \max_{i = 1, \dots , d} v_i\).
For \(0 < p < \infty\), a function $f$ and a measure $Q$, \(\lVert f \rVert_{L^p(Q)} \coloneqq (\int \lvert f \rvert^p dQ)^{1/p}\) represents the $L^p(Q)$ norm of $f$. If $Q$ is clear from the context, we write \(\lVert f \rVert_{L^p(Q)} = \lVert f \rVert_p\).
For \(p = \infty\), we define \(\lVert f \rVert_\infty \coloneqq \sup_{\bm{x}} f(\bm{x})\).
For two vector-valued functions $\bm{f}, \bm{g} : \mathcal{D} \to \mathbb{R}^d$, we define $\Verts{\bm{f} - \bm{g}}_\infty = \lVert\max_{j=1, \dots , d} \lvert f_j(\bm{x}) - g_j(\bm{x})\rvert \rVert_{\infty}$.
Operations and inequalities involving vectors, vector-valued functions, etc., are all element-wise operations and inequalities.
For example, for $\bm{x}, \bm{y} \in \mathbb{R}^d$, when $\bm{y} \neq 0$, we define $\bm{x} / \bm{y} \coloneqq (x_1 / y_1, \dots , x_d / y_d)^\top$.
Also, $\bm{x} \leq \bm{y}$ means that for all $i = 1, \dots , d$, $x_i \leq y_i$ holds.
For a vector $\bm{v} \in \mathbb{R}^d$ and a scalar \(a \in \mathbb{R}\), we define \(\bm{v} + a \coloneqq (a + v_1, \dots , a + v_d)^\top\).
For a vector $\bm{v} \in \mathbb{R}^d$ and a univariate function $g$, we define $g(\bm{v}) \coloneqq (g(v_1), \dots , g(v_d))^\top$.
For a real number $x \in \mathbb{R}$, $\lfloor x \rfloor$ is the largest integer less than or equal to $x$, and $\lceil x \rceil$ is the smallest integer greater than $x$.
For an integer $K$, we denote the set \(\{1, \dots , K\}\) as \([K]\).
We denote the $K$-dimensional standard basis vectors as $\bm{e}_1, \dots , \bm{e}_K$.
Each standard basis vector is represented like $\bm{e}_k = (0, \dots , 0, 1, 0, \dots , 0)^\top$.
The $(K-1)$-dimensional simplex is denoted by $\mathcal{S}^K$, i.e., $\mathcal{S}^K := \{\bm{v} \in \mathbb{R}^K \mid \sum_{k=1}^K v_k = 1, v_k \geq 0, k = 1, \dots , K\}$.
For two probability measures $P$ and $Q$, the Kullback-Leibler divergence $\mathrm{KL}(P \parallel Q)$ is defined as $\mathrm{KL}(P \parallel Q) \coloneqq \int \log (dP / dQ) dP$ when $P$ is absolutely continuous with respect to $Q$, and it is defined as $\mathrm{KL}(P \parallel Q) \coloneqq \infty$ otherwise.
For two sequences $(a_n)$ and $(b_n)$, if there exists a constant $C$ such that for all $n$, $a_n \leq Cb_n$, we write $a_n \lesssim b_n$. Furthermore, if $a_n \lesssim b_n$ and $b_n \lesssim a_n$, we write $a_n \asymp b_n$.

\subsection{Settings} \label{sec: settings}
Consider a multi-class classification problem with $K$ classes.
Let $\mathcal{X} \;= [0, 1]^d$ be the input space, and $\mathcal{Y} = \{\bm{e}_i\}_{i=1}^K$ be the set of labels.
Assume that the data $\pr{\bm{X}, \bm{Y}} \in \mathcal{X} \times \mathcal{Y}$ is generated from the following model:
\begin{equation}
    Y_k \mid \bm{X} = \bm{x} \sim \text{Bernoulli}(\eta_k(\bm{x})), \quad \bm{X} \sim P_{\bm{X}}, \quad k = 1, \dots , K, \label{eq:data-generating model}
\end{equation}
where $\eta_k(\bm{x}) \coloneqq \mathbb{P}\pr{\bm{Y} = \bm{e}_k \mid \bm{X} = \bm{x}}$ is the true conditional class probabilities, and $P_{\bm{X}}$ is the unknown distribution on the input space $\mathcal{X}$.
We denote the joint distribution of $\bm{X}$ and $\bm{Y}$ as $P$.
Let $\mathcal{D}_n = \{(\bm{X}_1, \bm{Y}_1),\dots,(\bm{X}_n, \bm{Y}_n)\}$ be an i.i.d. sample with size $n$ from the population distritbuion $P$.
The goal of the classification problem is to find a function $\bm{f}: \mathcal{X} \to \mathbb{R}^K$ (called the decision function) that predicts $\bm{Y}$ well when $\bm{X}$ are given.
However, the estimation of conditional class probabilities is also important in practice.
Here, we focus on the nonparametric estimation of conditional class probabilities.

In the estimation of conditional class probabilities, we typically consider the maximum likelihood estimation, i.e., we minimize the negative log-likelihood function.
Let $\bm{p}(\bm{x}) = (p_1(\bm{x}), \dots , p_K(\bm{x}))^\top$ be a model of the conditional class probability to estimate the true one $\bm{\eta}(\bm{x}) = (\eta_1(\bm{x}), \dots , \eta_K(\bm{x}))^\top$.
Given the data $\mathcal{D}_n$, 
the likelihood for the conditional class probability function $\bm{p}(\bm{x})$ is given by
$\prod_{i=1}^n\prod_{k=1}^K p_k(\bm{X}_i)^{Y_{ik}}$.
Here, $Y_{ik}$ is the $k$-th component of $\bm{Y}_i$.
The negative log-likelihood function is
\begin{equation}
    L(\bm{p}) \coloneqq -\frac{1}{n}\sum_{i=1}^n\sum_{k=1}^K Y_{ik}\log p_k(\bm{X}_i) = -\frac{1}{n}\sum_{i=1}^n \bm{Y}_i^\top \log \bm{p}(\bm{X}_i), \label{eq: negative likelihood function}
\end{equation}
and the maximum likelihood estimator (MLE) is defined as
\begin{equation}
    \hat{\bm{p}}_n \in \argmin_{\bm{p} \in \mathcal{F}_n} L(\bm{p}), \label{eq:mle}
\end{equation}
where $\mathcal{F}_n$ is a class of candidate functions.
Here, we note that the function $L$ is also known as 
the cross-entropy.
Throughout this paper, all estimators $\hat{\bm{p}}_{n} = (\hat{p}_{n,1}, \dots , \hat{p}_{n,K})^\top$ are considered as probability vectors for all $\bm{x} \in \mathcal{X}$, i.e., $p_k(\bm{x}) \geq 0$ for any $\bm{x} \in \mathcal{X}, k \in [K]$, and \(\sum_{k=1}^K p_k(\bm{x}) = 1\) for all $\bm{x} \in \mathcal{X}$.
This can be achieved using a DNN with the softmax function in the output layer.
The risk corresponding to the loss \eqref{eq: negative likelihood function} is given by
\begin{equation*}
    \mathbb{E}[L(\bm{p})] = \; \mathbb{E}_{\bm{X}}\sqbr{-\sum_{k=1}^K \eta_k(\bm{X})\log p_k(\bm{X})} = \mathbb{E}_{\bm{X}}\sqbr{-\boldsymbol{\eta}(\bm{X})^\top \log \bm{p}(\bm{X})},
\end{equation*}
where $\mathbb{E}_{\bm{X}}$ denotes the expectation with respect to the marginal distribution $P_{\bm{X}}$.
This risk is minimized when $\bm{p} = \boldsymbol{\eta}$, and the excess risk can be expressed as
\begin{equation}
    \mathbb{E}_{\bm{X}}\sqbr{\boldsymbol{\eta}(\bm{X})^\top \log \frac{\boldsymbol{\eta}(\bm{X})}{\hat{\bm{p}}_{n}(\bm{X})}} 
    = \mathbb{E}_{\bm{X}}\sqbr{\KL{\boldsymbol{\eta}(\bm{X})}{\hat{\bm{p}}_{n}(\bm{X})}}. \label{eq: KL excess risk}
\end{equation}
To show the consistency of the MLE $\hat{\bm{p}}_n$, it seems natural to evaluate the excess risk \eqref{eq: KL excess risk}.
However, as stated in Lemma 2.1 of \cite{BosSchmidt-Hieber22}, even when considering small classes as the model $\mathcal{F}_n$, the expected value of the excess risk \eqref{eq: KL excess risk} diverges.
For instance, if $\mathcal{F}_n$ includes all functions such as "piecewise constant with two or more pieces" or "piecewise linear with three or more pieces", the excess risk \eqref{eq: KL excess risk} diverges.
To address this problem, \cite{BosSchmidt-Hieber22} introduces the truncated KL divergence for $B > 0$:
\begin{equation*}
    \operatorname{KL}_B\pr{\boldsymbol{\eta}(\bm{X}) \parallel \hat{\bm{p}}_{n}(\bm{X})} 
    \coloneqq \boldsymbol{\eta}(\bm{X})^\top \pr{B \land \log \frac{\boldsymbol{\eta}(\bm{X})}{\hat{\bm{p}}_{n}(\bm{X})}},
\end{equation*}
and derives the convergence rates of the expected value of the truncated KL divergence:
\begin{equation}
    R_B(\boldsymbol{\eta}(\bm{X}), \hat{\bm{p}}_{n}(\bm{X})) \coloneqq \mathbb{E}_{\mathcal{D}_n, \bm{X}}\sqbr{\operatorname{KL}_B\pr{\boldsymbol{\eta}(\bm{X}) \parallel \hat{\bm{p}}_{n}(\bm{X})}}, \label{eq:truncated kl excess risk}
\end{equation}
where \(\mathbb{E}_{\mathcal{D}_n}\) denotes the expectation over the training data \(\mathcal{D}_n\).
Furthermore, they also provide convergence rates for the Hellinger distance based on the relationship between the truncated KL divergence and the Hellinger distance (see Lemma 3.4 of \cite{BosSchmidt-Hieber22}).
Here, for two probability measures $P, Q$ on the same measurable space, the squared Hellinger distance is defined as
\begin{equation*}
    H^2\pr{P, Q} \coloneqq \frac{1}{2} \int \pr{\sqrt{dP} - \sqrt{dQ}}^2.
\end{equation*}

In this paper, 
we employ a simpler general approach developed by \cite{vandeGeer00} to derive the convergence rates of the nonparametric MLE (NPMLE) in the logistic regression.
In contrast to \cite{BosSchmidt-Hieber22} and \cite{Bilodeau2023}, which essentially consider the truncated KL divergence, 
we directly evaluate the Hellinger distance between the true conditional probabilities and the NPMLE.
That is, instead of the original excess risk \eqref{eq: KL excess risk}, we evaluate the following risk:
\begin{equation}
    R(\boldsymbol{\eta}(\bm{X}), \hat{\bm{p}}_{n}(\bm{X})) 
    \coloneqq \mathbb{E}_{\bm{X}}\sqbr{H^2\pr{\boldsymbol{\eta}(\bm{X}), \hat{\bm{p}}_{n}(\bm{X})}}. \label{eq: Hellinger excess risk}
\end{equation}
Since the Hellinger distance is always bounded, we can avoid the divergence problem of the risk.
For two probability measures $P$ and $Q$ defined on the same measurable space, 
it holds that $2H^2(P, Q) \leq \KL{P}{Q}$.
Thus, the convergence in the Hellinger distance is a weaker result than the convergence in the KL divergence.
However, considering the convergence in terms of the Hellinger distance allows us to derive the convergence rates of the NPMLE based on the Hellinger distance under general conditions on the true conditional class probabilities $\bm{\eta}$.

In this study, for demonstrating oracle inequalities to evaluate the NPMLE for the multi-class classification problem, 
we assume the following relationship between the model $\mathcal{F}_n$ and the true conditional class probability $\boldsymbol{\eta}$:
\begin{assumption} \label{assumption}
    For a certain constant $c_0 > 0$, there exists a sequence $\Tilde{\bm{p}}_n \in \mathcal{F}_n$ such that for all $n \in \mathbb{N},\; \bm{x} \in \mathcal{X},\; k \in [K]$, the following holds:
    \begin{equation}
        \frac{\eta_k(\bm{x})}{\Tilde{p}_{n, k}(\bm{x})} \leq c_0^2. \label{eq:assumption}
    \end{equation}
\end{assumption}
This assumption means that, if choosing the model $\mathcal{F}_n$ appropriately, there is no need to impose additional conditions on the true conditional class probabilities $\boldsymbol{\eta}$.
We later show that this condition is satisfied by choosing a suitable family of DNNs as the model $\mathcal{F}_n$.

\subsection{Related works} \label{sec:related works}
This section describes related works and discusses their relevance and differences from this study.
Many previous studies impose strong conditions on the true class conditional probability $\boldsymbol{\eta}$ and the population distribution $P$.

First, we describe some existing results for the convergence of the classification error. 
In \cite{HuEtAl20}, convergence rates of the classification error for the DNN classifier that minimizes the hinge loss are derived.
They assume that $p(\bm{x}) - q(\bm{x})$ is representable by some DNN $f_n^* \in \mathcal{F}_n$, where \(p(x) \coloneqq dP\pr{X = x \mid Y = 1}/dQ, q(x) \coloneqq dP\pr{X = x \mid Y = -1} / dQ\) are the conditional densities with respect to a dominating measure \(Q\).

In \cite{KimEtAl21}, for binary classification using hinge loss, the Tsybakov noise condition as follows is imposed:
\begin{itemize}
    \item[(N)] There exist \(c_N > 0\) and \(q \in [0, \infty]\) such that for any \(t > 0\),
    \begin{equation*}
        \Pb\pr{\left|2\eta_1(\bm{X}) - 1\right| \leq t} \leq c_N t^q.
    \end{equation*}
\end{itemize}
For binary classification using logistic loss, the conditions are:
\begin{itemize}
    \item[(N\('\))] There exists a constant \(\eta_0 \in (0, 1)\) such that
    \begin{equation*}
        \Pb\pr{ \left|2\eta_1(\bm{X}) - 1\right| \leq \eta_0} = 0.
    \end{equation*}
    \item[(M\('\))] There exists a constant \(m_0 > 0\) such that
    \begin{equation*}
        \Pb\pr{\text{dist}(\bm{X}, D^*) \leq m_0} = 0.
    \end{equation*}
\end{itemize}
Here, \(D^* = \{\bm{x} : \eta_1(\bm{x}) = 1/2\}\), and \(\operatorname{dist}(x, D^*) = \inf_{\bm{x}^* \in D} \left\|\bm{x} - \bm{x}^*\right\|_2\).
In the case of multi-class classification with the multi-class hinge loss (\citealp{LeeEtAl04}), the following condition is assumed:
\begin{itemize}
    \item[(MN)] There exist \(C > 0\) and \(q \in [0, \infty]\) such that for any \(t > 0\),
    \begin{equation*}
        \Pb\pr{\tau(\bm{X}) \leq t} \leq C t^q,
    \end{equation*}
\end{itemize}
where \(\tau(\bm{x}) = \min_{k \neq C^*(\bm{x})} \left|\eta_k(\bm{x}) - \eta_{C^*(\bm{x})}(\bm{x})\right|\) with the Bayes classifier \(C^*\).

Next, we describe existing studies focusing on the estimation of conditional class probabilities.
As described above, the expected value of the KL divergence (\ref{eq: KL excess risk})
could easily diverge, we need some assumptions for the convergence of the excess risk.
\cite{OhnKim22} discuss convergence rates for the KL divergence for the regularized MLE in the binary classification problem.
They assume that the true logit function is uniformly bounded.
However, this assumption cannot handle situations where the true conditional class probabilities take extreme values (0 or 1) with a positive probability for some $k \in [K]$. That is, we need to assume that
\[
\exists c_L,c_U\in (0,1);\;
\Pb\big(\forall k \in [K];\;c_L < \eta_k(\bm{X}) < c_U\big) = 1.
\]

In \cite{BosSchmidt-Hieber22}, for the multi-class classification problem, they assume 
that the true conditional class probability function $\bm{\eta}$ belongs to the class of H\"{o}lder-smooth functions.
They consider the following \(\alpha\)-Small Value Bound (SVB) assumption for the true class conditional probabilities and derive the convergence rates of the risk with the truncated KL divergence \eqref{eq:truncated kl excess risk}.
\begin{Def}[Small Value Bound, Definition 3.1 of \cite{BosSchmidt-Hieber22}] \label{def: small value bound}
    Let \(\alpha \geq 0\) and \(\mathcal{H}\) be a function class. We say that \(\mathcal{H}\) is \(\alpha\)-small value bounded (or \(\alpha\)-SVB) if there exists a constant \(C > 0\), such that for all \(p = (p_1,\dots ,p_K)^\top \in \mathcal{H}\) it holds that
    \begin{equation*}
        \Pb\pr{p_k(\bm{X}) \leq t} \leq Ct^\alpha, \quad \text{for all \(t \in (0, 1]\) and all \(k \in [K]\)}.
    \end{equation*}
\end{Def}
Here, we note that the \(\alpha\)-SVB condition in Definition \ref{def: small value bound} with $\alpha = 0$ and  $C = 1$ always holds. 
In their results, we can see the interesting effect of $\alpha$ in the convergence rates.
Since the squared Hellinger distance is bounded by truncated KL divergence, 
the convergence rates of the risk with the truncated KL divergence can be considered as the rates of the risk with the Hellinger distance.
Although the convergence rates derived in this study are similar to theirs, we provide a more straightforward proof by directly evaluating the Hellinger distance.
Since the truncated KL divergence is not a commonly used metric, a well-established framework for evaluating it is not yet in place.
However, we can apply the general framework from \cite{vandeGeer00} for the evaluation of the Hellinger distance, allowing for a simpler proof.
Moreover, while their convergence rates require the sparsity of DNN models, our analysis allows the fully connected (dense) DNNs.

In \cite{Bilodeau2023}, to overcome the divergence problem of the KL divergence
the following modified estimator is considered:
\begin{equation*}
    \Tilde{\bm{p}}_n^{(1/n)} \coloneqq \frac{\hat{\bm{p}}_n + 1 / (nK)}{1 + 1 / n}.
\end{equation*}
They derive the convergence rates in the KL divergence for the modified estimator under the well-specified setting that the true conditional class probabilities are included in the given statistical model.
Using our oracle inequality in Section~\ref{sec:general}, similar results to those in \cite{Bilodeau2023} are obtained when the true conditional class probabilities are included in the given models.
In \cite{Bilodeau2023}, they use the probability inequality for the likelihood ratio from \cite{WongandShen1995}, and this inequality is obtained by evaluating the truncated KL divergence.
Although our results are for the Hellinger distance, 
our oracle inequality can be applied to the original NPMLE 
even for the unrealizable cases.
Since the true conditional class probability is not necessarily contained within the function class representable by a DNN, our results provide an evaluation under more realistic conditions.

\section{General theory for nonparametric logistic regression}\label{sec:general}
In this section, by using the general approach for theoretical analysis of NPMLEs developed by \cite{vandeGeer00}, 
we derive the oracle-type inequality for evaluating the expectation of the Hellinger distance \eqref{eq: Hellinger excess risk} between the true and estimated conditional class probabilities.

Before stating the theorem, several definitions and notations are introduced.
Given the class of conditional class probabilities $\mathcal{F}$ and $\Tilde{\bm{p}} \in \mathcal{F}$, with $0 < \delta \leq 1$, let $\Bar{\mathcal{F}}^{1/2}(\Tilde{\bm{p}}, \delta)$ denote the set of all functions that are the square root of arithmetic means of elements in $\mathcal{F}$ and $\Tilde{\bm{p}}$ such that $R(\cdot, \Tilde{\bm{p}}) \leq \delta^2$.
Formally,
\begin{equation*}
    \Bar{\mathcal{F}}^{1/2}(\Tilde{\bm{p}}, \delta) = \left\{\sqrt{\frac{\bm{p} + \Tilde{\bm{p}}}{2}} : R\left(\frac{\bm{p} + \Tilde{\bm{p}}}{2}, \Tilde{\bm{p}}\right) \leq \delta^2, \bm{p} \in \mathcal{F}\right\}.
\end{equation*}
Additionally, for \(p \geq 1\) and a measure \(Q\), the bracketing number for the $L^p(Q)$ metric of a function class $\mathcal{F}$ is defined as follows:
\begin{Def}[Bracketing Number for $L^p(Q)$ Metric]
    Let $N_{p, B}(\delta, \mathcal{F}, Q)$ be the smallest value of $N$ for which there exist pairs of functions $\{(\bm{f}_j^L, \bm{f}_j^U)\}_{j=1}^N$ such that $\|\bm{f}_j^U - \bm{f}_j^L\|_{L^p(Q)} \leq \delta$ for all $j \in [N]$, and such that for each $\bm{f} \in \mathcal{F}$, there is a $j = j(g) \in [N]$ such that
    \begin{equation*}
        \bm{f}_j^L \leq \bm{f} \leq \bm{f}_j^U.
    \end{equation*}
    We call $N_{p, B}(\delta, \mathcal{F}, Q)$ the $\delta$-bracketing number of $\mathcal{F}$.
\end{Def}

Using the two definitions introduced above, the bracketing entropy integral for the model \(\mathcal{F}_n\) and \(\Tilde{\bm{p}}_n\) in Assumption \ref{assumption} is defined as
\begin{equation*}
    J_B(\delta, \Bar{\mathcal{F}}_n^{1/2}(\Tilde{\bm{p}}_n, \delta), \mu) = \int_{\delta^2/(2^{13}c_0)}^\delta \sqrt{\log N_{2, B}(u, \Bar{\mathcal{F}}_n^{1/2}(\Tilde{\bm{p}}_n, \delta), \mu)} \; du \lor \delta.
\end{equation*}
Here, \(\mu\) represents the product measure of the counting measure on \([K]\) and \(P_{\bm{X}}\).
The quantity \(J_B(\delta, \Bar{\mathcal{F}}_n^{1/2}(\Tilde{\bm{p}}_n, \delta), \mu)\) represents a local complexity of the model in the neighborhood of \(\Tilde{\bm{p}}_n\).

The following theorem provides an oracle inequality that evaluates the Hellinger distance between the estimator \(\hat{\bm{p}}_n\) and the true conditional class probability function \(\boldsymbol{\eta}\).
\begin{theorem}[Oracle inequality] \label{thm:probability oracle inequality}
    Consider the K-class classification model \eqref{eq:data-generating model}. Let $\hat{\bm{p}}_n$ be given in \eqref{eq:mle}. Suppose that Assumption \ref{assumption} is satisfied. Take $\Psi(\delta) \geq J_B(\delta, \Bar{\mathcal{F}}_n^{1/2}(\Tilde{\bm{p}}_n, \delta), \mu)$ in such a way that $\Psi(\delta)/\delta^2$ is a non-increasing function of $\delta$. Then, for universal constant $c$ and for
    \begin{equation}
        \sqrt{n}\delta_n^2 \geq c\Psi(\delta_n), \label{eq:critical inequality}
    \end{equation}
    we have for all $\delta \geq \delta_n$,
    \begin{equation}
        \Pb\left(R(\hat{\bm{p}}_n, \boldsymbol{\eta}) > 514(1 + c_0^2)(\delta^2 + R(\Tilde{\bm{p}}_n, \boldsymbol{\eta}))\right) \leq c \exp\left(-\frac{n\delta^2}{c^2}\right). \label{eq:probability oracle inequality}
    \end{equation}
\end{theorem}

Furthermore, we obtain the following oracle inequality for the expectation by integrating both sides of \eqref{eq:probability oracle inequality}.

\begin{corollary} \label{cor:expectation oracle inequality}
    Under the conditions of Theorem \ref{thm:probability oracle inequality}, we have
    \begin{equation}
        \mathbb{E}_{\mathcal{D}_n}\left(R(\hat{\bm{p}}_n, \boldsymbol{\eta})\right) \leq 514(1 + c_0^2)(\delta_n^2 + R(\Tilde{\bm{p}}_n, \boldsymbol{\eta})) + \frac{c^3}{n}, \label{eq:expectation oracle inequality}
    \end{equation}
    where $\mathbb{E}_{\mathcal{D}_n}$ denotes the expectation over the training data $\mathcal{D}_n$.
\end{corollary}

The proofs of Theorem \ref{cor:expectation oracle inequality} and Corollary \ref{cor:expectation oracle inequality} are in Appendix \ref{app: proof of the oracle inequality}

\begin{remark}
    Generally, evaluating the covering number is easier than evaluating the bracketing number.
    Therefore, in practical applications of this theorem, Lemma \ref{lem:bracketing and covering with sup norm} is often used to bound the bracketing number with the covering number, and the evaluation is conducted using the covering number.
    Lemma \ref{lem:bracketing and covering with sup norm} can be found in Appendix \ref{app: proof of the convergence rate}.
\end{remark}

\begin{remark}
    Theorem \ref{thm:probability oracle inequality} and Corollary \ref{cor:expectation oracle inequality} show that convergence rates of the NPMLE can be derived if the complexity and approximation error of the model \(\mathcal{F}_n\) can be evaluated. To see this concretely, in addition to Theorem \ref{thm:convergence rates with non-sparse DNN}, Appendix \ref{app: convergence rates for Besov space} includes additional results for the case in which the true conditional class probabilities belong to the anisotropic Besov space (see \citealp{Suzuki&Nitanda2021}).
\end{remark}

\begin{remark}
    In Theorem \ref{thm:probability oracle inequality} and Corollary \ref{cor:expectation oracle inequality}, the choice of $\tilde{\bm{p}}_n$ is arbitrary as long as it satisfies Assumption \ref{assumption}. 
    For instance, if the considered model includes a function $\bm{p}$ such that $\exists c>0;\;\forall \bm{x} \in \mathcal{X};\;\bm{p}(\bm{x}) > c$, 
    then Assumption \ref{assumption} is automatically satisfied.
    However, a poorly chosen $\tilde{\bm{p}}_n$ that satisfies (\ref{eq:assumption}) 
    in Assumption \ref{assumption} may lead to a large value of the approximation error term $R(\Tilde{\bm{p}}_n, \boldsymbol{\eta})$, which can negatively affect the derived convergence rates.
    To establish tight convergence rates, 
    it is essential to select a sequence ${\tilde{\bm{p}}_n}$ carefully.
    In particular, for deriving convergence rates of the NPMLE with deep learning, 
    we demonstrate that there exists a sequence $\tilde{\bm{p}}_n$ within the DNN models that satisfies Assumption \ref{assumption} while ensuring a sufficiently small approximation error (Lemma \ref{lem:approximation error for composition structured probability}).
    By substituting this choice of $\tilde{\bm{p}}_n$ into Corollary \ref{cor:expectation oracle inequality}, we can obtain tight convergence rates.
\end{remark}

\section{Nonparametric logistic regression with Deep learning}
In this section, based on Corollary \ref{cor:expectation oracle inequality}, we derive convergence rates of the NPMLE with the DNN model \(\mathcal{F}_n\) for the true class conditional probabilities \(\boldsymbol{\eta}\) under the composition assumption (\citealp{Schmidt-Hieber20}).
Moreover, we show that the NPMLE with the DNN model can achieve optimal convergence rates up to \(\log (n)\)-factors.

\subsection{Deep Neural Network model}
Firstly, we mathematically define the deep feedfoward neural network model.
To define the DNN model, we need to specify the activation function $\sigma: \mathbb{R} \to \mathbb{R}$, and here we use the ReLU function defined as $\sigma(x) = \max(x, 0)$.
For $\bm{v} = (v_1, \dots, v_r)^\top \in \mathbb{R}^r$, the activation function with bias $\sigma_{\bm{v}}: \mathbb{R}^r \to \mathbb{R}^r$ is defined as
\begin{equation*}
    \sigma_{\bm{v}}
    \begin{pmatrix}
    y_1 \\
    \vdots \\
    y_r
    \end{pmatrix}
    =
    \begin{pmatrix}
    \sigma(y_1 - v_1) \\
    \vdots \\
    \sigma(y_r - v_r)
    \end{pmatrix}.
\end{equation*}

The structure of the feedforward neural network is determined by a positive integer \(L\), referred to as the number of hidden layers or depth, a width vector \(\bm{m} = (m_0, \dots , m_{L+1}) \in \mathbb{N}^{L+2}\), and the activation function \(\boldsymbol{\psi}\) of the output layer.

A deep neural network with \((L, \bm{m}, \boldsymbol{\psi})\) refers to any function of the form
\begin{equation}
    f: \mathbb{R}^{m_0} \to \mathbb{R}^{m_{L+1}},\; \bm{x} \mapsto f(\bm{x}) = \boldsymbol{\psi}W_L\sigma_{\bm{v}_L}W_{L-1}\sigma_{\bm{v}_{L-1}}\cdots W_1\sigma_{\bm{v}_1}W_0\bm{x}, \label{eq:DNN function}
\end{equation}
where \(W_i \in \mathbb{R}^{m_{i+1} \times m_i}\) is a weight matrix, and \(\bm{v}_i \in \mathbb{R}^{m_i}\) is a bias vector.
For the regression problem, the identity function, denoted by \(\mathrm{id}: \bm{x} \mapsto \bm{x}\), is commonly used as the activation function \(\boldsymbol{\psi}\) for the output layer.
For the classification problem, the softmax function
\begin{equation}
    \boldsymbol{\Psi} : \mathbb{R}^{m_{L+1}} \to \mathbb{R}^{m_{L+1}},\; \bm{x} \mapsto \pr{\frac{e^{x_1}}{\sum_{j=1}^{m_{L+1}}e^{x_j}}, \dots , \frac{e^{x_{m_{L+1}}}}{\sum_{j=1}^{m_{L+1}}e^{x_j}}}
\end{equation}
is frequently employed.
In the multi-class classification setting of this study, we set 
$m_0 = d$, $m_{L+1} = K$, and $\boldsymbol{\psi} = \boldsymbol{\Psi}$.

The space of all networks with parameter magnitudes restricted is denoted as
\begin{equation*}
    \mathcal{F}_{\boldsymbol{\psi}}(L, \bm{m}, B) = \left\{f \text{ of the form \eqref{eq:DNN function}}: \max_{j = 0, \dots , L} \lVert W_j \rVert_\infty \lor \lvert \bm{v}_j \rvert_\infty \leq B\right\},
\end{equation*}
where $\bm{v}_0$ is a vector with all components being 0.
Note that \(\mathcal{F}_{\boldsymbol{\psi}}(L, \bm{m}, B)\) is not subject to any constraints other than the number of layers, width, and the magnitude of parameters. In other words, \(\mathcal{F}_{\boldsymbol{\psi}}(L, \bm{m}, B)\) also includes fully connected (non-sparse) DNNs.
Moreover, we define the space of sparse networks as follows:
\begin{equation*}
    \mathcal{F}_{\boldsymbol{\psi}}(L, \bm{m}, B, s) = \left\{f \in \mathcal{F}_{\boldsymbol{\psi}}(L, \bm{m}, B): \sum_{j=0}^L (\lVert W_j \rVert_0 + \lvert \bm{v}_j \rvert_0) \leq s)\right\}.
\end{equation*}
Here, $\lVert W_j \rVert_0$ and $\lvert \bm{v}_j \rvert_0$ denote the number of non-zero components of $W_j$ and $\bm{v}_j$ respectively.

\subsection{Convergence rates} \label{sec: convergence rates}

In this section, we utilize Corollary \ref{cor:expectation oracle inequality} to derive the specific convergence rates in the Hellinger distance.
The class of the true conditional class probability function $\boldsymbol{\eta}$ is essential for the convergence rates.
We assume that $\boldsymbol{\eta}$ are composition structured functions as introduced in \cite{Schmidt-Hieber20}.
We define composite structured functions as follows.

We assume that each component \(\eta_k, k \in [K]\) of the true conditional class probabilities is a composite function of multiple functions, that is, \(\eta_k = g_r \circ g_{r-1} \circ \cdots \circ g_1 \circ g_0,\)
where \(g_i: [a_i, b_i]^{d_i} \to [a_{i+1}, b_{i+1}]^{d_{i+1}}\).
We denote the components of \(g_i\) as \(g_i = (g_{ij})_{j=1, \dots , d_{i+1}}^\top\), and let \(t_i\) be the maximum number of variables on which each \(g_{ij}\) depends.
Consequently, each \(g_{ij}\) is a \(t_i\)-variate function.

For a given $\beta > 0$ and $D \subset \mathbb{R}^m$, 
the ball of $\beta$-H\"older functions with radius $Q > 0$ is defined as
\begin{equation*}
    C^\beta(D, K) = \left\{f: D \to \mathbb{R}: \sum_{\bm{\alpha}: \lvert \bm{\alpha} \rvert < \beta} \lVert \partial^{\bm{\alpha}} f \rVert_\infty + \sum_{\bm{\alpha}: \lvert \bm{\alpha} \rvert = \lfloor \beta \rfloor} \sup_{\bm{x}, \bm{y} \in D,\; \bm{x} \neq \bm{y}} \frac{\lvert \partial^{\bm{\alpha}}f(\bm{x}) - \partial^{\bm{\alpha}}f(\bm{y}) \rvert}{\lvert \bm{x} - \bm{y} \rvert_\infty^{\beta - \lfloor \beta \rfloor}} \leq Q\right\}.
\end{equation*}

We assume that each function \(g_{ij}\) is a \(\beta\)-Hölder continuous function.
Since \(g_{ij}\) is a \(t_i\)-variate function, it follows that \(g_{ij} \in C^{\beta_i}([a_i, b_i]^{t_i}, Q_i)\) for some \(a_i, b_i\), and \(Q_i\).
Therefore, the function class underlying \(\eta_k\) is given by
\begin{align*}
    \mathcal{G}(r, \bm{d}, \bm{t}, \boldsymbol{\beta}, Q) \coloneqq 
    \big\{\eta_k = g_r \circ \cdots \circ g_0 \mid\, &g_i = (g_{ij})_j: [a_i, b_i]^{d_i} \to [a_{i+1}, b_{i+1}]^{d_{i+1}},
    \\
    &g_{ij} \in C^{\beta_i}([a_i, b_i]^{t_i}, Q), \text{for some} \lvert a_i \rvert, \lvert b_i \rvert \leq Q \big\},
\end{align*}
where \(\bm{d} \coloneqq (d_0, \dots , d_{r+1}),\; \bm{t} \coloneqq (t_0, \dots , t_r),\) and \(\boldsymbol{\beta} \coloneqq (\beta_0, \dots , \beta_r)\).
Consequently, the function class underlying the true conditional class probabilities \(\boldsymbol{\eta}\) is given by
\begin{equation*}
    \mathcal{G}^\prime(r, \bm{d}, \bm{t}, \boldsymbol{\beta}, Q, K) \coloneqq \br{\boldsymbol{\eta} = \pr{\eta_1, \dots , \eta_K}^\top : [0, 1]^d \to \mathcal{S}^K \mid \eta_k \in \mathcal{G}(r, \bm{d}, \bm{t}, \boldsymbol{\beta}, Q), k \in [K]}.
\end{equation*}
To ensure that \(\mathcal{G}^\prime(r, \bm{d}, \bm{t}, \boldsymbol{\beta}, Q, K)\) is non-empty, we assume that \(KQ \geq 1\).
As an important example, the class \(\mathcal{G}^\prime(r, \bm{d}, \bm{t}, \boldsymbol{\beta}, Q, K)\) includes the following generalized additive models:
\begin{equation}
    \eta_k(x_1, \dots , x_d) = \Psi_k\left( \sum_{i=1}^d f_{ki}(x_i) \right), \label{eq: an example of composition structured functions}
\end{equation}
where \(\Psi_k(\bm{x}) = e^{x_k} / \sum_{i=1}^d e^{x_i}\).
In this model, each \(f_{ki}\) can be considered as a ``score'' that reflects how each covariate \(x_i\) influences the confidence that the class is \(k\). 

The following theorem provides a bound on the expected squared Hellinger distance between the true and estimated conditional class probabilities for the nonparametric logistic regression with deep learning.

\begin{theorem} [Convergence rates] \label{thm:convergence rates with non-sparse DNN}
    Consider the K-class classification model \eqref{eq:data-generating model} for the true conditional class probabilities $\boldsymbol{\eta}$ in the class $\mathcal{G}^\prime(r, \bm{d}, \bm{t}, \boldsymbol{\beta}, Q, K)$. Let $\hat{\bm{p}}_n$ be given in \eqref{eq:mle} with $\mathcal{F}_n = \mathcal{F}_{\boldsymbol{\Psi}}(L, \bm{m}, B)$ satisfying the following conditions:
    \[
    \text{\rm (i) }\; L \asymp \log (n),\quad
    \text{\rm (ii) }\; \sqrt{n\phi_n} \asymp \min_{i=1, \dots , L}m_i,\quad
    \text{\rm (iii)}\; B \asymp \max_{i=0, \dots, r} (n\phi_n)^{\frac{2\beta_i + 2}{t_i}}.
    \]
    Then, there exists a constant $C$ depending only on $q, \bm{m}, \bm{t}, \boldsymbol{\beta}, Q, K$ and $\bm{d}$ such that
    \begin{equation*}
        \mathbb{E}_{\mathcal{D}_n}\sqbr{R(\hat{\bm{p}}_n, \boldsymbol{\eta})} \leq C\phi_n\log(n)^3
    \end{equation*}
    for sufficient large \(n\),
    where
    \begin{equation*}
        \beta_i^* := \beta_i\prod_{l=i+1}^r(\beta_{l} \land 1)
        \;\text{ and }\;
        \phi_n := \max_{i=0, \dots , r}n^{-\frac{\beta_i^*}{\beta_i^* + t_i}}.
    \end{equation*}
\end{theorem}

\begin{remark}
    If a DNN is chosen as the conditions of Theorem \ref{thm:convergence rates with non-sparse DNN}, Assumption \ref{assumption} for Corollary \ref{cor:expectation oracle inequality} is automatically satisfied (for details, see the proof of Theorem \ref{thm:convergence rates with non-sparse DNN} in Appendix \ref{app: proof of the convergence rate}).
\end{remark}
\begin{remark}
    The convergence rates of the non-sparse DNN model presented in Theorem \ref{thm:convergence rates with non-sparse DNN} are the same as those of the sparse DNN model in \cite{BosSchmidt-Hieber22}. Since sparse feedforward neural networks are not commonly used in practice, in this sense, our results can be considered closer to a more realistic setting.
\end{remark}
\begin{remark}
    It has been demonstrated in \cite{BosSchmidt-Hieber22} that imposing the SVB (Small Value Bound) condition in Definition \ref{def: small value bound} can lead to faster convergence rates. 
    However, verifying that the SVB condition holds in practice is challenging. 
    Therefore, in this study, we focus on analyzing the convergence rates without imposing additional restrictive conditions.
    On the other hand, imposing the SVB condition can lead to faster convergence rates, even in our analysis. 
    Specifically, under the \(\alpha\)-SVB assumption, the convergence rate of the maximum likelihood estimator using deep learning becomes \(\max_{i = 0, \dots, r} n^{-(1 + \alpha)\beta_i^* / ((1 + \alpha)\beta_i^* + t_i)}\). 
    These convergence rates are the same as shown in \cite{BosSchmidt-Hieber22}. 
    Here, we note that the SVB condition with $\alpha = 0$ always holds.
    Substituting $\alpha = 0$ into the convergence rates results in the same rates as that obtained when the SVB condition is not imposed.
\end{remark}
\begin{remark} \label{rem:convergence rates for K}
    In practical classification problems, the number of classes \(K\) is usually fixed.
    Then, it has been included in the constant \(C\) in the bound of Theorem \ref{thm:convergence rates with non-sparse DNN}.
    The convergence rates including \(K\) in Theorem \ref{thm:convergence rates with non-sparse DNN} are
    \begin{equation*}
        \max_{i=0, \dots , r}K^{\frac{2\beta_i^* + 3t_i}{\beta_i^* + t_i}}n^{-\frac{\beta_i^*}{\beta_i^* + t_i}}.
    \end{equation*}
    These convergence rates are slightly slower with respect to \(K\) compared to those of \cite{BosSchmidt-Hieber22}, but since it uses a non-sparse network, it is a more general setting.
    However, as discussed later, these convergence rates are not the minimax optimal rates with respect to \(K\).
\end{remark}

\begin{remark}
    Since Corollary \ref{cor:expectation oracle inequality} does not assume any specific structure for the true function class, it is applicable not only to the composite structured functions but also to other function classes.
    If one can verify that Assumption \ref{assumption} is satisfied for these classes and evaluate the approximation error \(R\pr{\Tilde{\bm{p}}_n, \boldsymbol{\eta}}\), it is possible to derive convergence rates for DNNs in other settings, similar to Theorem \ref{thm:convergence rates with non-sparse DNN}.
\end{remark}

Theorem \ref{thm:convergence rates with non-sparse DNN} indicates that the convergence rates do not directly depend on the dimensionality $d$ of the input space.
Therefore, when $t_{i^*} < d$ for $i^* = \argmax_{i=0, \dots , r}n^{-\beta_i^* /(\beta_i^* + t_i)}$, deep learning can avoid the curse of dimensionality even in the classification problem.
Indeed, considering the model (\ref{eq: an example of composition structured functions}), if we assume that for each $k$ and $i$, \(f_{ki} \in \mathcal{C}^{\beta}([0, 1], Q)\), then by defining \(\bm{g}_{k0}(\bm{x}) = (f_{k1}(x_1), \dots, f_{kd}(x_d))^\top\) and \(g_{k1}(\bm{y}) = \sum_{j=1}^d y_j\), we can express \(\eta_k\) as \(\eta_k = \Psi_k \circ g_{k1} \circ \bm{g}_{k0}\).
For any \(\gamma > 1\), it holds that \(g_{k1} \in \mathcal{C}^{\gamma}([-Q, Q]^d, (Q + 1)d)\) and \(\Psi_k \in \mathcal{C}^{\gamma}([-Qd, Qd], V)\), where \(V\) is a sufficiently large constant depending on \(\gamma\).
Thus, we have
\begin{equation*}
    \bm{\eta} \in \mathcal{G}^\prime(2, (d, d, 1, 1), (1, d, 1), (\beta, (\beta \lor 1)d, 1), (Q + 1)d \lor V, K).
\end{equation*}
By Theorem \ref{thm:convergence rates with non-sparse DNN}, the convergence rates of the maximum likelihood estimator using deep learning with respect to the Hellinger distance are \(n^{-\beta / (\beta + 1)}\). Notably, these rates do not depend on the input dimension \(d\), indicating that the curse of dimensionality is avoided.

\subsection{Minimax Optimality}
In this section, we demonstrate that the convergence rates of the estimator obtained in the previous section are minimax optimal, up to a logarithmic factor. The following theorem presents this result.

\begin{theorem} \label{thm:minimax}
    Consider the K-class classification model \eqref{eq:data-generating model} with \(X_i\) drawn from a distribution with Lebesgue density on \([0,1]^d\), which is lower and upper bounded by positive constants. For any nonnegative integer \(r\), any dimension vectors \(\bm{d}\) and \(\bm{t}\) satisfying \(t_i \leq \min(d_0 , \cdots , d_{i-1})\) for all \(i\), any smoothness vector \(\boldsymbol\beta\), and all sufficiently large constants \(Q > 0\), there exists a positive constant \(c\) such that
    \begin{equation*}
        \inf_{\hat{\bm{p}}_n}\sup_{\boldsymbol{\eta} \in \mathcal{G}^\prime\pr{r, \bm{d}, \bm{t}, \boldsymbol \beta, Q, K}} \mathbb{E}_{\mathcal{D}_n}\left[R\pr{\hat{\bm{p}}_n, \boldsymbol{\eta}}\right] \geq c\phi_n,
    \end{equation*}
    where infimum is taken over all estimators.
\end{theorem}

From this theorem, it follows that the NPMLE using DNNs achieves the minimax optimal convergence rate, up to a logarithmic factor when the underlying function class of the true conditional class probabilities is \(\mathcal{G}^\prime\pr{r, \bm{d}, \bm{t}, \boldsymbol \beta, Q, K}\).

\begin{remark}
    As mentioned in Remark \ref{rem:convergence rates for K}, we consider the number of classes \(K\) as a fixed constant.
    Therefore, Theorem \ref{thm:minimax} does not provide the minimax optimal convergence rate with respect to \(K\).
    In fact, for the case of \(r = 0\), where \(\mathcal{G}^\prime\) can be viewed as a Hölder class, the minimax convergence rate is \(K^{2d_0 / (d_0 + \beta_0)}n^{-\beta_0 / (d_0 + \beta_0)}\) according to Example 3 in \cite{Bilodeau2023}.
    However, the lower bound of the minimax convergence rate shown in Theorem \ref{thm:minimax} is \(K^{-1}n^{-\beta_0 / (d_0 + \beta_0)}\).
    In this study, we do not conduct a more detailed analysis for minimax lower bounds as the number of classes \(K\) is not emphasized. 
    However, there is a possibility that using the framework from \cite{Bilodeau2023} could allow for the demonstration of a more refined minimax lower bound for \(K\).
\end{remark}

\begin{appendix}
\setcounter{theorem}{0}
\section{Proofs}
In this section, we provide the proofs for the theorems in the main text.
Throughout this section, \(P_n\) represents the empirical distribution of the data \(\mathcal{D}_n\), and we use the shorthand notation \(Qf = \int f dQ\) for measurable functions \(f\) and measures \(Q\).

\subsection{Proof of Theorem \ref{thm:probability oracle inequality}} \label{app: proof of the oracle inequality}
In this section, we prove Theorem \ref{thm:probability oracle inequality}.
To achieve this, we present several lemmas.
The key lemma for the proof of Theorem \ref{thm:probability oracle inequality} is the following.

\begin{lemma}[Basic Inequality] \label{lem:basic inequality}
    For MLE \(\hat{\bm{p}}_n\) in \eqref{eq:mle} and any conditional class probabilities \(\Tilde{\bm{p}}_n \in \mathcal{F}_n\), we have
    \begin{equation*}
        R\left(\frac{\hat{\bm{p}}_n + \Tilde{\bm{p}}_n}{2}, \Tilde{\bm{p}}_n\right) \leq (P_n - P)\pr{\frac{1}{2}\bm{Y}^\top \log\pr{\frac{\hat{\bm{p}}_n + \Tilde{\bm{p}}_n}{2\Tilde{\bm{p}}_{n}}\mathbbm{1}(\Tilde{\bm{p}}_n > 0)}} + 2(1 + c_0)\sqrt{R\pr{\frac{\hat{\bm{p}}_n + \Tilde{\bm{p}}_n}{2}, \Tilde{\bm{p}}_n}R\pr{\Tilde{\bm{p}}_n, \boldsymbol{\eta}}}.
    \end{equation*}
\end{lemma}
\begin{proof}
    Using the concavity of the logarithm, for any \(\bm{x} \in \mathcal{X}\), we have
    \begin{equation*}
        \frac{1}{2}\log\left(\frac{\hat{\bm{p}}_n(\bm{x}) + \Tilde{\bm{p}}_{n}(\bm{x})}{2\Tilde{\bm{p}}_n(\bm{x})}\right)\mathbbm{1}(\Tilde{\bm{p}}_n > 0) \geq \frac{1}{4}\log\left(\frac{\hat{\bm{p}}_n(\bm{x})}{\Tilde{\bm{p}}_n(\bm{x})}\right)\mathbbm{1}(\Tilde{\bm{p}}_n > 0).
    \end{equation*}
    Combining this with the fact that \(\hat{\bm{p}}_n\) is the MLE, we get
    \begin{align*}
        0 &\leq P_n\pr{\frac{1}{4}\bm{Y}^\top \log\pr{\frac{\hat{\bm{p}}_n}{\Tilde{\bm{p}}_n}}\mathbbm{1}(\Tilde{\bm{p}}_n > 0)} 
        \leq P_n\pr{\frac{1}{2}\bm{Y}^\top \log\pr{\frac{\hat{\bm{p}}_n + \Tilde{\bm{p}}_n}{2\Tilde{\bm{p}}_n}}\mathbbm{1}(\Tilde{\bm{p}}_n > 0)} \\
        &= (P_n - P)\pr{\frac{1}{2}\bm{Y}^\top \log\pr{\frac{\hat{\bm{p}}_n + \Tilde{\bm{p}}_n}{2\Tilde{\bm{p}}_n}}\mathbbm{1}(\Tilde{\bm{p}}_n > 0)} +  P\pr{\frac{1}{2}\bm{Y}^\top \log\pr{\frac{\hat{\bm{p}}_n + \Tilde{\bm{p}}_n}{2\Tilde{\bm{p}}_n}}\mathbbm{1}(\Tilde{\bm{p}}_n > 0)} \\
        &\leq (P_n - P)\pr{\frac{1}{2}\bm{Y}^\top \log\pr{\frac{\hat{\bm{p}}_n + \Tilde{\bm{p}}_n}{2\Tilde{\bm{p}}_n}}\mathbbm{1}(\Tilde{\bm{p}}_n > 0)} - P\pr{\bm{Y}^\top\pr{1 - \sqrt{\frac{\hat{\bm{p}}_n + \Tilde{\bm{p}}_n}{2\Tilde{\bm{p}}_n}}}\mathbbm{1}(\Tilde{\bm{p}}_n > 0)},
    \end{align*}
    where we used \(\log(x) \leq x - 1\) in the last step.
    Now,
    \begin{align*}
        &P\pr{\bm{Y}^\top\pr{1 - \sqrt{\frac{\hat{\bm{p}}_n + \Tilde{\bm{p}}_n}{2\Tilde{\bm{p}}_n}}\mathbbm{1}(\Tilde{\bm{p}}_n > 0)}} 
        = \int_\mathcal{X}\int_\mathcal{Y} \bm{y}^\top\pr{1 - \sqrt{\frac{\hat{\bm{p}}_n(\bm{x}) + \Tilde{\bm{p}}_n(\bm{x})}{2\Tilde{\bm{p}}_n(\bm{x})}}\mathbbm{1}(\Tilde{\bm{p}}_n > 0)} dP(\bm{y} \mid \bm{x}) dP_X(\bm{x}) \\
        &= \int_\mathcal{X} \sum_{k=1}^K\pr{1 - \sqrt{\frac{\hat{p}_{n, k}(\bm{x}) + \Tilde{p}_{n, k}(\bm{x})}{2\Tilde{p}_{n, k}(\bm{x})}}}\Tilde{p}_{n, k}(\bm{x})\mathbbm{1}(\Tilde{\bm{p}}_n > 0) dP_X(\bm{x}) \\
        &\quad+ \int_\mathcal{X} \sum_{k=1}^K\pr{1 - \sqrt{\frac{\hat{p}_{n, k}(\bm{x}) + \Tilde{p}_{n, k}(\bm{x})}{2\Tilde{p}_{n, k}(\bm{x})}}}\pr{\eta_k(\bm{x}) - \Tilde{p}_{n, k}(\bm{x})}\mathbbm{1}(\Tilde{\bm{p}}_n > 0) dP_X(\bm{x}) \\
        &= \int_\mathcal{X} H^2\pr{\frac{\hat{\bm{p}}_n + \Tilde{\bm{p}}_n}{2}, \Tilde{\bm{p}}_n} dP_X(\bm{x}) \\
        &\quad+ \int_\mathcal{X} \sum_{k=1}^K\pr{\sqrt{\Tilde{p}_{n, k}} - \sqrt{\frac{\hat{p}_{n, k} + \Tilde{p}_{n, k}}{2}}}\pr{\sqrt{\eta_k} - \sqrt{\Tilde{p}_{n, k}}}\pr{1 + \sqrt{\frac{\eta_k}{\Tilde{p}_{n, k}}}}\mathbbm{1}(\Tilde{\bm{p}}_n > 0) dP_X(\bm{x}) \\
        &\geq R\pr{\frac{\hat{\bm{p}}_n + \Tilde{\bm{p}}_n}{2}, \Tilde{\bm{p}}_n} -2(1 + c_0)\int_\mathcal{X} H\pr{\frac{\hat{\bm{p}}_n + \Tilde{\bm{p}}_n}{2}, \Tilde{\bm{p}}_n}H\pr{\Tilde{\bm{p}}_n, \boldsymbol{\eta}} dP_X(\bm{x}) \\
        &\geq R\pr{\frac{\hat{\bm{p}}_n + \Tilde{\bm{p}}_n}{2}, \Tilde{\bm{p}}_n} - 2(1 + c_0)\sqrt{R\pr{\frac{\hat{\bm{p}}_n + \Tilde{\bm{p}}_n}{2}, \Tilde{\bm{p}}_n}R\pr{\Tilde{\bm{p}}_n, \boldsymbol{\eta}}},
    \end{align*}
    where we used the Cauchy-Schwarz inequality in the first inequality and H\"older's inequality in the last inequalities.
    This completes the proof.
\end{proof}

\begin{lemma}\label{lem:excess risk of convex combination}
    For all conditional class probabilities \(\bm{p}, \bm{q}\), we have
    \begin{equation*}
        R(\bm{p}, \bm{q}) \leq 16R\left(\frac{\bm{p} + \bm{q}}{2}, \bm{q}\right).
    \end{equation*}
\end{lemma}
\begin{proof}
    For any \(k \in [K]\), we have
    \begin{align*}
        \left(\sqrt{\frac{p_k(\bm{x}) + q_k(\bm{x})}{2}} + \sqrt{q_k(\bm{x})}\right)^2 &= \frac{p_k(\bm{x}) + q_k(\bm{x})}{2} + q_k(\bm{x}) + 2\sqrt{q_k(\bm{x})}\sqrt{\frac{p_k(\bm{x}) + q_k(\bm{x})}{2}} \\
        &\leq \frac{p_k(\bm{x}) + q_k(\bm{x})}{2} + q_k(\bm{x}) + q_k(\bm{x}) + \frac{p_k(\bm{x}) + q_k(\bm{x})}{2} \\
        &\leq 4\left(p_k(\bm{x}) + q_k(\bm{x})\right).
    \end{align*}
    The inequality in the second line follows from \(a^2 + b^2 \geq 2ab\). Taking square roots on both sides, we get
    \begin{equation*}
        \sqrt{\frac{p_k(\bm{x}) + q_k(\bm{x})}{2}} + \sqrt{q_k(\bm{x})} 
        \leq 2\left\{\sqrt{p_k(\bm{x})} + \sqrt{q_k(\bm{x})}\right\}.
    \end{equation*}
    Therefore, we have
    \begin{align*}
        \left\vert\sqrt{p_k(\bm{x})} - \sqrt{q_k(\bm{x})}\right\vert 
        &= \frac{\left\vert p_k(\bm{x}) - q_k(\bm{x})\right\vert}{\sqrt{p_k(\bm{x})} + \sqrt{q_k(\bm{x})}} 
        = 2\frac{\sqrt{\frac{p_k(\bm{x}) + q_k(\bm{x})}{2}} + \sqrt{q_k(\bm{x})}}{\sqrt{p_k(\bm{x})} + \sqrt{q_k(\bm{x})}} \left\vert\sqrt{\frac{p_k(\bm{x}) + q_k(\bm{x})}{2}} - \sqrt{q_k(\bm{x})}\right\vert \\
        &\leq 4\left\vert\sqrt{\frac{p_k(\bm{x}) + q_k(\bm{x})}{2}} - \sqrt{q_k(\bm{x})}\right\vert.
    \end{align*}
    This implies
    \begin{align*}
        H^2(\bm{p}, \bm{q}) &= \frac{1}{2}\sum_{k=1}^K \left\vert\sqrt{p_k(\bm{x})} - \sqrt{q_k(\bm{x})}\right\vert^2 
        \leq 16 \cdot \frac{1}{2}\sum_{k=1}^K \left\vert\sqrt{\frac{p_k(\bm{x}) + q_k(\bm{x})}{2}} - \sqrt{q_k(\bm{x})}\right\vert^2 
        = 16H^2\left(\frac{\bm{p} + \bm{q}}{2}, \bm{q}\right).
    \end{align*}
    Thus, \(R\) satisfies the inequality by definition.
\end{proof}

From Lemma \ref{lem:basic inequality} and Lemma \ref{lem:excess risk of convex combination}, it is understood that to evaluate \(R\left(\pr{\hat{\bm{p}}_n + \Tilde{\bm{p}}_n} / 2, \Tilde{\bm{p}}_n\right)\), one needs to evaluate the empirical process \(\nu_n(g_{\bm{p}}) \coloneqq \sqrt{n}(P_n - P)g_{\bm{p}}\), indexed by
\begin{equation*}
    \mathcal{G}_n \coloneqq \br{g_{\bm{p}} = \frac{1}{2}\bm{y}^\top \log\pr{\frac{\bm{p} + \Tilde{\bm{p}}_n}{2\Tilde{\bm{p}}_n}}\mathbbm{1}(\Tilde{\bm{p}}_n > 0) \mid \bm{p} \in \mathcal{F}_n},
\end{equation*}
which appears on the right-hand side of Lemma \ref{lem:basic inequality}.
Note that, since \(g_{\bm{p}} \geq 1 / 2\log\pr{\Tilde{\bm{p}}_n / 2\Tilde{\bm{p}}_n} = -1 / 2\log(2)\), \(\mathcal{G}_n\) is bounded from below.
To evaluate this empirical process, we employ a Bernstein-type uniform inequality. 
Before stating this uniform inequality, let us introduce some definitions. 

For a function \(g : \mathcal{X} \times \mathcal{Y} \to \mathbb{R}\) and \(M > 0\), the Bernstein semi-norm \(\rho_M(g)\) is defined as
\begin{equation*}
    \rho_M^2(g) = 2M^2 \int \pr{e^{\frac{\lvert g(\bm{x}, \bm{y}) \rvert}{M}} - 1 - \frac{\lvert g(\bm{x}, \bm{y}) \rvert}{M}} dP(\bm{x}, \bm{y}).
\end{equation*}
Furthermore, for functions \(f, g : \mathcal{X} \times \mathcal{Y} \to \mathbb{R}\), \(\rho_M(f - g)\) is referred to as the Bernstein difference between \(f\) and \(g\).
We then define the generalized bracketing number using this Bernstein difference, which is as follows:
\begin{Def}[Generalized Bracketing Number]
    Let \(\mathcal{N}_{B, M}\pr{\delta, \mathcal{F}, P}\) be the smallest value of \(N\) for which there exist pairs of functions \(\{(f_j^L, f_j^U)\}_{j=1}^N\) such that \(\rho_M(f_j^U - f_j^L) \leq \delta\) for all \(j \in [N]\), and such that for each \(f \in \mathcal{F}\), there is a \(j = j(g) \in [N]\) such that
    \begin{equation*}
        f_j^L \leq f \leq f_j^U.
    \end{equation*}
    (Take \(\mathcal{N}_{B, M}\pr{\delta, \mathcal{F}, P} = \infty\) if no finite set of such brackets exists.) We call \(\mathcal{N}_{B, M}\pr{\delta, \mathcal{F}, P}\) the \(\delta\)-generalized bracketing number of \(\mathcal{F}\).
\end{Def}
The following is a Bernstein-type uniform inequality.
\begin{theorem}[Theorem 5.11 of \cite{vandeGeer00}] \label{thm:uniform bernstein inequality}
    Let \(\mathcal{G}\) be a collection of functions.
    Assume that \(\mathcal{G}\) satisfies
    \begin{equation}
        \sup_{g \in \mathcal{G}} \rho_M(g) \leq U \label{eq:bernstein condition}
    \end{equation}
    for some \(U > 0,\; M > 0\). Take
    \begin{align}
        a &\leq C_1\sqrt{n}U^2/M \label{eq:condition 1 for bernstein}, \\
        a &\leq 8\sqrt{n}U \label{eq:condition 2 for bernstein}, \\
        a &\geq C_0\left(\int_{a/(2^6\sqrt{n})}^U \sqrt{\log\mathcal{N}_{B, M}(u, \mathcal{G}, P)} \; du \lor U\right) \label{eq:condition 3 for bernstein}, \\
        C_0^2 &\geq C^2(C_1 + 1) \label{eq:condition 4 for bernstein}.
    \end{align}
    Then,
    \begin{equation*}
        \mathbb{P}\left(\sup_{g \in \mathcal{G}} \lvert \nu_n(g) \rvert \geq a\right) \leq C\exp\left(-\frac{a^2}{C^2(C_1 + 1)U^2}\right).
    \end{equation*}
\end{theorem}

Furthermore, for the proof of Theorem \ref{thm:probability oracle inequality}, we present several lemmas to verify the conditions \eqref{eq:bernstein condition}-\eqref{eq:condition 4 for bernstein} of Theorem \ref{thm:uniform bernstein inequality}.

\begin{lemma}[Lemma 7.1 in \cite{vandeGeer00}] \label{lem:bound for bernstein norm}
    For \(x \geq - \Gamma\) with $\Gamma > 0$, we have
    \begin{equation*}
        2(e^{\lvert x \rvert} - 1 - \lvert x \rvert) \leq c_\Gamma^2(e^x - 1)^2,\quad
        c_\Gamma = \frac{2(e^\Gamma - 1 - \Gamma)}{(e^{-\Gamma} - 1)^2}.
    \end{equation*}
\end{lemma}

\begin{lemma} \label{lem:bernstein norm and hellinger risk}
    Let Assumption \ref{assumption} hold.
    For all \(g_{\bm{p}} \in \mathcal{G}_n\), we have
    \begin{equation*}
        \rho_1^2(g_{\bm{p}}) \leq
        16c_0^2R\left(\frac{\bm{p} +\Tilde{\bm{p}}_n}{2},\Tilde{\bm{p}}_n\right).
    \end{equation*}
\end{lemma}
\begin{proof}
    Since \(\mathcal{G}_n\) is bounded from below, we apply Lemma \ref{lem:bound for bernstein norm} with \(x = g_{\bm{p}}\) and \(\Gamma = \frac{1}{2}\log(2)\).
    Noting that \(\bm{y}\) is a standard basis vector, we have
    \begin{align*}
        \rho_1^2(g_{\bm{p}}) &= \int_{\mathcal{X} \times \mathcal{Y}}\left(e^{\lvert g_{\bm{p}} \rvert} - 1 - \lvert g_{\bm{p}} \rvert\right) dP(\bm{x}, \bm{y}) \\
        &\leq 8\int_\mathcal{X}\int_\mathcal{Y}(e^{g_{\bm{p}}} - 1)^2 dP(\bm{y} \mid \bm{x}) dP_{\bm{X}}(\bm{x}) \\
        &= 8\int_\mathcal{X}\sum_{k=1}^K\pr{\sqrt{\frac{p_k + \Tilde{p}_{n, k}}{2\Tilde{p}_{n, k}}}\mathbbm{1}(\Tilde{\bm{p}} > 0) - 1}^2 \Tilde{p}_{n, k}\frac{\eta_k}{\Tilde{p}_{n, k}} dP_X(\bm{x}) \\
        &\leq 16c_0^2\int_\mathcal{X}H^2\left(\frac{\bm{p} +\Tilde{\bm{p}}_n}{2},\Tilde{\bm{p}}_n\right) dP_X(\bm{x}) 
        = 16c_0^2R\left(\frac{\bm{p} +\Tilde{\bm{p}}_n}{2},\Tilde{\bm{p}}_n\right).
    \end{align*}
\end{proof}

\begin{lemma} \label{lem:bracket for bernstein norm and hellinger risk}
    Let Assumption \ref{assumption} hold.
    Let \(0 \leq \bm{p}^L \leq \bm{p}^U\).
    Also, let
    \begin{equation*}
        g^L = \frac{1}{2}\bm{y}^\top\log\pr{\frac{\bm{p}^L + \Tilde{\bm{p}}_n}{2\Tilde{\bm{p}}_n}}\mathbbm{1}(\Tilde{\bm{p}}_n > 0),\text{ and }\; g^U = \frac{1}{2}\bm{y}^\top\log\pr{\frac{\bm{p}^U + \Tilde{\bm{p}}_n}{2\Tilde{\bm{p}}_n}}\mathbbm{1}(\Tilde{\bm{p}}_n > 0).
    \end{equation*}
    Then
    \begin{equation*}
        \rho_1(g^U - g^L) \leq 2c_0\sqrt{R\pr{\frac{\bm{p}^L + \Tilde{\bm{p}}_n}{2}, \frac{\bm{p}^U + \Tilde{\bm{p}}_n}{2}}}.
    \end{equation*}
\end{lemma}
\begin{proof}
    Applying Lemma \ref{lem:bound for bernstein norm} with \(\Gamma \downarrow 0\) and noting that \(\bm{y}\) is the standard basis vector, we have
    \begin{align*}
        \rho_1^2(g^U - g^L) &= \int_{\mathcal{X} \times \mathcal{Y}}\pr{e^{g^U - g^L} - 1 - \pr{g^U - g^L}} dP(\bm{x}, \bm{y}) \\
        &\leq \int_\mathcal{X}\int_\mathcal{Y} \pr{e^{g^U - g^L} - 1}^2 dP(\bm{y} \mid \bm{x}) dP_{\bm{X}}(\bm{x}) \\
        &= \int_\mathcal{X} \sum_{k=1}^K\pr{\frac{\sqrt{p_k^U + \Tilde{p}_{n, k}} - \sqrt{p_k^L + \Tilde{p}_{n, k}}}{\sqrt{p_k^L + \Tilde{p}_{n, k}}}\mathbbm{1}(\Tilde{\bm{p}} > 0)}^2\Tilde{p}_{n, k}\frac{\eta_k}{\Tilde{p}_{n, k}} dP_{\bm{X}}(\bm{x}) \\
        &\leq 2c_0^2\int_\mathcal{X}\sum_{k=1}^K\pr{\sqrt{\frac{p_k^U + \Tilde{p}_{n, k}}{2}} - \sqrt{\frac{p_k^L + \Tilde{p}_{n, k}}{2}}}^2 dP_{\bm{X}}(\bm{x}) \\
        &= 4c_0^2R\pr{\frac{\bm{p}^L + \Tilde{\bm{p}}_n}{2}, \frac{\bm{p}^U + \Tilde{\bm{p}}_n}{2}}.
    \end{align*}
\end{proof}

Now that the preparations are complete, we present the proof of Theorem \ref{thm:probability oracle inequality}.
\begin{proof}[Proof of Theorem \ref{thm:probability oracle inequality}]
    From Lemma \ref{lem:basic inequality}, we have
    \begin{align*}
        R\left(\frac{\hat{\bm{p}}_n + \Tilde{\bm{p}}_n}{2}, \Tilde{\bm{p}}_n\right) &\leq 
        (P_n - P)\pr{\frac{1}{2}\bm{Y}^\top \log\pr{\frac{\hat{\bm{p}}_n + \Tilde{\bm{p}}_n}{2\Tilde{\bm{p}}}}} + 2(1 + c_0)\sqrt{R\pr{\frac{\hat{\bm{p}}_n + \Tilde{\bm{p}}_n}{2}, \Tilde{\bm{p}}_n}R\pr{\Tilde{\bm{p}}_n, \boldsymbol{\eta}}} \\
        &= \mathrm{I} + \mathrm{II}.
    \end{align*}
    If $\mathrm{I} \leq \mathrm{II}$, then
    \begin{equation*}
        R\left(\frac{\hat{\bm{p}}_n + \Tilde{\bm{p}}_n}{2}, \Tilde{\bm{p}}_n\right) \leq 4(1 + c_0)\sqrt{R\pr{\frac{\hat{\bm{p}}_n + \Tilde{\bm{p}}_n}{2}, \Tilde{\bm{p}}_n}R\pr{\Tilde{\bm{p}}_n, \boldsymbol{\eta}}}.
    \end{equation*}
    Hence,
    \begin{equation*}
        R\left(\frac{\hat{\bm{p}}_n + \Tilde{\bm{p}}_n}{2}, \Tilde{\bm{p}}_n\right) \leq 16(1 + c_0)^2R\pr{\Tilde{\bm{p}}_n, \boldsymbol{\eta}}.
    \end{equation*}
    By the triangle inequality and Lemma \ref{lem:excess risk of convex combination},
    \begin{align*}
        R\pr{\hat{\bm{p}}_n, \boldsymbol{\eta}} &\leq 2\pr{R\pr{\hat{\bm{p}}_n, \Tilde{\bm{p}}_n} + R\pr{\Tilde{\bm{p}}_n, \boldsymbol{\eta}}}
        \leq 2\pr{16R\left(\frac{\hat{\bm{p}}_n + \Tilde{\bm{p}}_n}{2}, \Tilde{\bm{p}}_n\right) + R\pr{\Tilde{\bm{p}}_n, \boldsymbol{\eta}}} \\
        &\leq 514(1 + c_0)^2R\pr{\Tilde{\bm{p}}_n, \boldsymbol{\eta}}.
    \end{align*}
    This concludes the proof.
    Next, consider the case where $\mathrm{I} > \mathrm{II}$.
    If $\mathrm{I} > \mathrm{II}$,
    \begin{equation*}
        R\left(\frac{\hat{\bm{p}}_n + \Tilde{\bm{p}}_n}{2}, \Tilde{\bm{p}}_n\right) \leq 2(P_n - P)\pr{\frac{1}{2}\bm{Y}^\top \log\pr{\frac{\hat{\bm{p}}_n + \Tilde{\bm{p}}_n}{2\Tilde{\bm{p}}}}} = 2\frac{\nu_n(g_{\hat{\bm{p}}_n})}{\sqrt{n}}.
    \end{equation*}
    Combining this with Lemma \ref{lem:excess risk of convex combination} and the triangle inequality,
    \begin{align*}
        &\mathbb{P}\pr{R\pr{\hat{\bm{p}}_n, \boldsymbol{\eta}} > 2(1 + c_0^2)\delta^2 + 2R\pr{\Tilde{\bm{p}}_n, \boldsymbol{\eta}}}
        \leq \mathbb{P}(R\pr{\hat{\bm{p}}_n, \Tilde{\bm{p}}_n} > (1 + c_0^2)\delta^2) \\
        \leq& \mathbb{P}\,\left(\sup_{\bm{p} \in \mathcal{F}_n, R((\bm{p} + \Tilde{\bm{p}}_n) / 2, \Tilde{\bm{p}}_n) > \delta^2 / 16} \nu_n(g_{\bm{p}}) - \frac{\sqrt{n}}{2}R\left(\frac{\bm{p} + \Tilde{\bm{p}}_n}{2}, \Tilde{\bm{p}}_n\right) \geq 0\right).
    \end{align*}
    Since the Hellinger distance is bounded, $R\left(\pr{\bm{p} + \Tilde{\bm{p}}_n} / 2, \Tilde{\bm{p}}_n\right) \leq 1$.
    Then, let $S = \min\{s : 2^{2s + 2}\delta^2 / 16 > 1\}$, we have
    \begin{align*}
        &\mathbb{P}\,\left(\sup_{\bm{p} \in \mathcal{F}_n, R((\bm{p} + \Tilde{\bm{p}}_n) / 2, \Tilde{\bm{p}}_n) > \delta^2 / 16} \nu_n(g_{\bm{p}}) - \frac{\sqrt{n}}{2}R\left(\frac{\bm{p} + \Tilde{\bm{p}}_n}{2}, \Tilde{\bm{p}}_n\right) \geq 0\right) \\
        \leq& \sum_{s=0}^S\mathbb{P}\,\left(\sup_{\bm{p} \in \mathcal{F}_n, 2^{2s}\delta^2 / 16 < R((\bm{p} + \Tilde{\bm{p}}_n) / 2, \Tilde{\bm{p}}_n) \leq 2^{2s+2}\delta^2 / 16} \nu_n(g_{\bm{p}}) - \frac{\sqrt{n}}{2}R\left(\frac{\bm{p} + \Tilde{\bm{p}}_n}{2}, \Tilde{\bm{p}}_n\right) \geq 0\right) \\
        \leq& \sum_{s=0}^S\mathbb{P}\,\left(\sup_{\bm{p} \in \mathcal{F}_n, R((\bm{p} + \Tilde{\bm{p}}_n) / 2, \Tilde{\bm{p}}_n) \leq 2^{2s+2}\delta^2 / 16} \nu_n(g_{\bm{p}}) \geq \sqrt{n}2^{2s}\frac{\delta^2}{2^5}\right).
    \end{align*}
    To apply Theorem \ref{thm:uniform bernstein inequality}, we need to check the conditions for each $s$. Note that Assumption~\ref{assumption} holds for all $c_0^\prime > c_0$, and we can choose $c_0$ such that $c_0 \geq 1$.
    \begin{itemize}
        \item By Lemma \ref{lem:bernstein norm and hellinger risk}, if $R\left(\pr{\bm{p} + \Tilde{\bm{p}}_n} / 2, \Tilde{\bm{p}}_n\right) \leq 2^{2s+2}\delta^2 / 16$, then $\rho_1(g_{\bm{p}}) \leq 2^{s+1}c_0\delta$.
        Thus, choosing $U = 2^{s+1}c_0\delta$ satisfies \eqref{eq:bernstein condition}.
        \item Setting $a = \sqrt{n}2^{2s}\delta^2 / 2^5$, if $C_1 = 15 \geq (2^7c_0^2)^{-1}$, then \eqref{eq:condition 1 for bernstein} holds.
        \item To satisfy \eqref{eq:condition 2 for bernstein}, it is enough to ensure $\delta \leq 2^{-S+9} \Leftrightarrow 2^{S}\delta / 4 \leq 2^7$. This condition is satisfied since the left side does not exceed one according to the definition of $S$.
        \item Choose $c = \sqrt{2}2^8c_0C$. By Lemma \ref{lem:bracket for bernstein norm and hellinger risk}, we have
        \begin{align*}
            &\sqrt{\int_\mathcal{X}\sum_{k=1}^K\pr{\sqrt{\frac{p_k^U + \Tilde{p}_{n, k}}{2}} - \sqrt{\frac{p_k^L + \Tilde{p}_{n, k}}{2}}}^2 dP_{\bm{X}}(\bm{x})}
            = \sqrt{2}\sqrt{\int_\mathcal{X}H^2\left(\frac{\bm{p}^U + \Tilde{\bm{p}}_n}{2}, \frac{\bm{p}^U + \Tilde{\bm{p}}_n}{2}\right)dP_{\bm{X}}(\bm{x})}\\
            &= \sqrt{2}\sqrt{R\left(\frac{\bm{p}^U + \Tilde{\bm{p}}_n}{2}, \frac{\bm{p}^U + \Tilde{\bm{p}}_n}{2}\right)} 
            \geq \frac{1}{\sqrt{2}c_0}\rho_1\left(g^U - g^L\right).
        \end{align*}
        Thus, for any $u > 0$,
        \begin{equation*}
            H_{2, B}\left(u, \Bar{\mathcal{F}}^{1/2}(\Tilde{\bm{p}}_n, \delta), \mu\right) \geq \mathcal{H}_{B, 1}\pr{\sqrt{2}c_0u, \br{g_{\bm{p}} \in \mathcal{G}_n : R\pr{\frac{\bm{p} + \Tilde{\bm{p}}}{2}, \Tilde{\bm{p}}_n} \leq \delta^2}, P}.
        \end{equation*}
        Since $\sqrt{n}\delta_n^2 \geq c\Psi(\delta_n) \geq cJ_B(\delta_n, \Bar{\mathcal{F}}^{1/2}(\Tilde{\bm{p}}_n, \delta_n), \mu)$,
        \begin{align*}
            \sqrt{n}\delta_n^2 &\geq c\Psi(\delta_n) \geq c \int_{\delta_n^2/(2^{13}c_0)}^{\delta_n} \sqrt{\log N_{2, B}\pr{u, \Bar{\mathcal{F}}^{1/2}(\Tilde{\bm{p}}_n, \delta_n), \mu}} du \\
            &\geq c \int_{\delta_n^2/(2^{13}c_0)}^{\delta_n} \sqrt{\log\mathcal{N}_{B, 1}\pr{\sqrt{2}c_0u, \br{g_{\bm{p}} \in \mathcal{G}_n : R\pr{\frac{\bm{p} + \Tilde{\bm{p}}}{2}, \Tilde{\bm{p}}_n} \leq \delta^2}, P}} du \\
            &= \frac{c}{\sqrt{2}c_0} \int_{\sqrt{2}\delta_n^2/2^{13}}^{\sqrt{2}c_0\delta_n} \sqrt{\log\mathcal{N}_{B, 1}\pr{u, \br{g_{\bm{p}} \in \mathcal{G}_n : R\pr{\frac{\bm{p} + \Tilde{\bm{p}}}{2}, \Tilde{\bm{p}}_n} \leq \delta^2}, P}} du .
        \end{align*}
        Here, since \(2^{s+1}\delta_n/\sqrt{2} \geq \delta_n\), replacing \(\delta_n\) with \(2^{s+1}\delta_n/\sqrt{2}\) in the above expression is valid, and
        \begin{align*}
            \sqrt{n}2^{2s+1}\delta_n^2 &\geq \frac{c}{\sqrt{2}c_0} \int_{\sqrt{2}2^{2s}\delta_n^2/(2^{12})}^{2^{s+1}c_0\delta_n} \sqrt{\mathcal{N}_{B, 1}\pr{u, \br{g_{\bm{p}} \in \mathcal{G}_n : R\pr{\frac{\bm{p} + \Tilde{\bm{p}}}{2}, \Tilde{\bm{p}}_n} \leq 2^{2s+2}\delta_n^2 / 2}, P}} du \\
            &\geq \frac{c}{\sqrt{2}c_0} \int_{2^{2s}\delta_n^2/(2^{11})}^{2^{s+1}c_0\delta_n} \sqrt{\mathcal{N}_{B, 1}\pr{u, \br{g_{\bm{p}} \in \mathcal{G}_n : R\pr{\frac{\bm{p} + \Tilde{\bm{p}}}{2}, \Tilde{\bm{p}}_n} \leq 2^{2s+2}\delta_n^2 / 16}, P}} du .
        \end{align*}
        Thus, for \(\delta \geq \delta_n\),
        \begin{align*}
            a &\geq \frac{c}{2^6\sqrt{2}c_0} \int_{a/\left(2^6\sqrt{n}\right)}^{U} \sqrt{\mathcal{N}_{B, 1}\pr{u, \br{g_{\bm{p}} \in \mathcal{G}_n : R\pr{\frac{\bm{p} + \Tilde{\bm{p}}}{2}, \Tilde{\bm{p}}_n} \leq 2^{2s+2}\delta^2 / 16}, P}} du \\
            &= 4C \int_{a/\left(2^6\sqrt{n}\right)}^{U} \sqrt{\mathcal{N}_{B, 1}\pr{u, \br{g_{\bm{p}} \in \mathcal{G}_n : R\pr{\frac{\bm{p} + \Tilde{\bm{p}}}{2}, \Tilde{\bm{p}}_n} \leq 2^{2s+2}\delta^2 / 16}, P}} du.
        \end{align*}
        Also, since \(\sqrt{n}\delta_n^2 \geq c\Psi(\delta_n) \geq c\delta_n\), it follows that \(\sqrt{n}2^{2s}\delta_n^2/2^5 \geq c\delta_n2^{2s}/ (2^5c_0)\), implying that for \(\delta \geq \delta_n\), \(a \geq 4C2^{s+1}\delta\) holds.
        Therefore, choosing \(C_0 = 4C\) satisfies \eqref{eq:condition 3 for bernstein}.
        \item Since \eqref{eq:condition 4 for bernstein} also holds, Theorem \ref{thm:uniform bernstein inequality} is applicable, and for all \(\delta \geq \delta_n\),
        \begin{align*}
            &\mathbb{P}\Big(R\pr{\hat{\bm{p}}_n, \boldsymbol{\eta}} > 514(1 + c_0^2)(\delta^2 + R\pr{\Tilde{\bm{p}}_n, \boldsymbol{\eta}})\Big) 
            \leq \mathbb{P}\Big(R\pr{\hat{\bm{p}}_n, \boldsymbol{\eta}} > 2(1 + c_0^2)\delta^2 + 2R\pr{\Tilde{\bm{p}}_n, \boldsymbol{\eta}}\Big)\\
            &\leq \sum_{s=0}^\infty C\exp\pr{-\frac{n2^{2s}\delta^2}{2^{16}C^2c_0^2}} 
            \leq \sum_{s=1}^\infty C\exp\pr{-\frac{n\delta^2s}{2^{16}C^2c_0^2}}
            \leq C\pr{1 - \exp\pr{-\frac{n\delta^2}{2^{16}C^2c_0^2}}}^{-1}\exp\pr{-\frac{n\delta^2}{2^{16}C^2c_0^2}} \\
            &\leq C\pr{1 - \exp\pr{-2}}^{-1}\exp\pr{-\frac{2n\delta^2}{c^2}}
            \leq c\exp\pr{-\frac{n\delta^2}{c^2}}.
        \end{align*}
        Here, the definition of \(c\) and \eqref{eq:critical inequality} are used.
    \end{itemize}
\end{proof}

\begin{proof}[Proof of Cororally \ref{cor:expectation oracle inequality}]
    For a non-negative random variable \(X\), the equality \(\mathbb{E}\sqbr{X} = \int_0^\infty \mathbb{P}\pr{X \geq t} dt\) holds.
    Using this fact and Theorem 3.1, we have
    \begin{align*}
        \mathbb{E}_{\mathcal{D}_n}\sqbr{R\pr{\hat{\bm{p}}_n, \boldsymbol{\eta}}} &= \int_0^\infty \mathbb{P}\pr{R\pr{\hat{\bm{p}}_n, \boldsymbol{\eta}} \geq t} dt
        \leq 514(1 + c_0^2)\pr{\delta_n^2 + R\pr{\Tilde{\bm{p}}_n, \boldsymbol{\eta}}} + c\int_0^\infty \exp\pr{-\frac{nt}{c^2}} dt \\
        &= 514(1 + c_0^2)\pr{\delta_n^2 + R\pr{\Tilde{\bm{p}}_n, \boldsymbol{\eta}}} + \frac{c^3}{n}.
    \end{align*}
\end{proof}

\subsection{Proof of Theorem \ref{thm:convergence rates with non-sparse DNN}} \label{app: proof of the convergence rate}
In this section, we will prove Theorem \ref{thm:convergence rates with non-sparse DNN}. First, we will state several lemmas related to approximation error and use them to prove Theorem \ref{thm:convergence rates with non-sparse DNN}. Throughout this section, for any function class \(\mathcal{F}\), we denote \(\mathcal{F}^{1/2} \coloneqq \{\sqrt{f} : f \in \mathcal{F}\}\).
Furthermore, when considering the composition of NNs and similar operations, we employ the enlarging, composition, depth synchronization, and parallelization properties as outlined in Section A.1 of \cite{BosSchmidt-Hieber22}.

First, we define the covering number and present several lemmas that are necessary for the proof of Theorem \ref{thm:convergence rates with non-sparse DNN}.
\begin{Def}[Covering Number for Supremum Norm]
    Let \(N(\delta, \mathcal{F}, \Verts{\; \cdot \;}_\infty)\) be the smallest value of \(N\) for which there exist \(\{g_j\}_{j=1}^n\) such that
    \begin{equation}
        \sup_{g \in \mathcal{G}} \min_{j \in [N]} \Verts{g - g_j}_\infty \leq \delta.
    \end{equation}
    We call \(N(\delta, \mathcal{F}, \Verts{\; \cdot \;}_\infty)\) the \(\delta\)-covering number of \(\mathcal{G}\)
\end{Def}

\begin{lemma}\label{lem:bracketing and covering with sup norm}
    Let \(\mathcal{G}\) be a function class from \(\mathcal{X}\) to \(\mathbb{R}^d\). Suppose \(Q\) is a finite measure on \(\mathcal{G}\) with \(Q(\mathcal{G}) = q < \infty\). Then, for all \(\delta > 0\),
    \begin{equation*}
        N_{p, B}\pr{\delta, \mathcal{G}, Q} \leq N\pr{\frac{\delta}{2q^{1/p}}, \mathcal{G}, \Verts{\;\cdot\;}_\infty}.
    \end{equation*}
\end{lemma}
\begin{proof}
    Let \(\{g_j\}_{j=1}^N\) be the smallest set satisfying
    \begin{equation*}
        \sup_{g \in \mathcal{G}}\min_{j = 1, \dots , N} \Verts{g - g_j}_\infty \leq \frac{\delta}{2q^{1/p}}.
    \end{equation*}
    Then, for any \(g \in \mathcal{G}\), there exists some \(j \in \br{1, \dots , N}\) such that for all \(x \in \mathcal{X}\),
    \begin{equation*}
        g_j(x) - \frac{\delta}{2q^{1/p}} \leq g(x) \leq g_j(x) + \frac{\delta}{2q^{1/p}}.
    \end{equation*}
    Therefore, for this \(j\), if we define \(g_j^L \coloneqq g_j - \frac{\delta}{2q^{1/p}}\) and \(g_j^U \coloneqq g_j + \frac{\delta}{2q^{1/p}}\), then
    \begin{equation*}
        \Verts{g_j^U - g_j^L}_{p, Q} = \pr{\int \frac{\delta^p}{q} dQ}^{\frac{1}{p}} = \delta,
    \end{equation*}
    and thus, \(\{(g_j^L, g_j^U)\}_{j=1}^N\) forms a \(\delta\)-bracketing set for \(\mathcal{G}\). Consequently,
    \begin{equation*}
        N_{p, B}\pr{\delta, \mathcal{G}, Q} \leq N\pr{\frac{\delta}{2q^{1/p}}, \mathcal{G}, \Verts{\;\cdot\;}_\infty}.
    \end{equation*}
\end{proof}

\begin{lemma}[Lemma 3 of \cite{suzuki19}] \label{lem:covering number of DNN}
    For any \(\delta > 0\),
    \begin{equation*}
        \log N\pr{\delta, \mathcal{F}_{\mathrm{id}}\pr{L, \bm{m}, B, s}, \Verts{\;\cdot\;}_\infty} \leq 2sL\log((B \lor 1)(\lvert \bm{m} \rvert_\infty + 1)) + s\log(\delta^{-1}L).    
    \end{equation*}
\end{lemma}

\begin{lemma}[Theorem 2 of \cite{KohlerLanger21}] \label{lem:approximation error for hölder}
    For any function \(f \in \mathcal{C}^\beta \pr{[0, 1]^d, Q}\) and integer \(N\) satisfies
    \begin{equation*}
        N \geq 2 \quad \text{and} \quad N^{2p} \geq c\left(\max_{\bm{\alpha} < \beta}\{\lVert \partial^{\bm{\alpha}} f \rVert_\infty\} \lor 1\right)^{4(\lfloor \beta \rfloor + 1)}
    \end{equation*}
    for some sufficiently large constant \(c \geq 1\), there exists a neural network
        $\Tilde{f} \in \mathcal{F}_\textup{id}\pr{L, (d, m, \dots , m, 1)}$
    with depth
    \begin{equation*}
        L \geq 5 + \lceil \log_4 (N^{\frac{2\beta_i}{d}}) \rceil (\lceil \log_2(\lfloor \beta \rfloor \lor d + 1) \rceil + 1)
    \end{equation*}
    and width
    \begin{equation*}
        m \geq 2^{d + 6} \binom{d + \lfloor \beta \rfloor}{d} d^2 (\lfloor \beta \rfloor + 1) N,
    \end{equation*}
    such that
    \begin{equation*}
        \lVert \Tilde{f} - f \rVert_\infty \leq C \left(\max_{\bm{\alpha} < \beta}\{\lVert \partial^{\bm{\alpha}} f \rVert_\infty\} \lor 1\right)^{4(\lfloor \beta \rfloor + 1)} N^{-\frac{2\beta}{d}},
    \end{equation*}
    where \(C\) is a constant only depending on \(\beta, Q, d\).
\end{lemma}

\begin{lemma}[Lemma 3 of \cite{Schmidt-Hieber20}] \label{lem:upper bound for composite function}
    Let \(f = \bm{g}_r \circ \cdots \circ \bm{g}_0, g_{ij} \in \mathcal{C}^{\beta_i}\pr{[a_i, b_i]^{t_i}, Q_i}\), and assume that \(Q_i \geq 1\).
    For \(i = 1, \dots , r-1\), define
    \begin{equation*}
        \bm{h}_0 \coloneqq \frac{\bm{g}_0(\bm{x})}{2Q_0} + \frac{1}{2},\;\; \bm{h}_i(\bm{x}) \coloneqq \frac{\bm{g}_i\pr{2Q_{i-1}\bm{x} - Q_{i-1}}}{2Q_i} + \frac{1}{2},\;\; \bm{h}_r(\bm{x}) \coloneqq \bm{g}_r\pr{2Q_{r-1}\bm{x} - Q_{r-1}}.
    \end{equation*}
    Then, we have
    \begin{equation*}
        f = \bm{g}_r \circ \cdots \circ \bm{g}_0 = \bm{h}_r \circ \cdots \circ \bm{h}_0,
    \end{equation*}
    where \(h_{0j}\) takes values in \([0, 1], h_{0j} \in \mathcal{C}^{\beta_0}\pr{[0, 1]^{t_0}, 1}, h_{ij} \in \mathcal{C}^{\beta_i}\pr{[0, 1]^{t_i}, \pr{2Q_{i-1}}^{\beta_i}}\) for \(i = 1, \dots , r - 1\) and \(h_{rj} \in \mathcal{C}^{\beta_r}\pr{[0, 1]^{t_r}, Q_r\pr{2Q_{r-1}}^{\beta_r}}\).
    Also, for any functions \(\Tilde{\bm{h}}_i = \pr{\Tilde{h}_{ij}}_j^\top\) with \(\Tilde{h}_{ij} : [0, 1]^{t_i} \to [0, 1]\),
    \begin{equation*}
        \Verts{\bm{h}_r \circ \cdots \circ \bm{h}_0 - \Tilde{\bm{h}}_r \circ \cdots \circ \Tilde{\bm{h}}_0}_\infty \leq Q_r\prod_{l = 0}^{r-1}\pr{2Q_l}^{\beta_{l+1}}\sum_{i=1}^r\Verts{\bm{h}_i - \Tilde{\bm{h}}_i}^{\prod_{l=i+1}^r \beta_l \land 1}.
    \end{equation*}
\end{lemma}

\begin{lemma}[Theorem 1 of \cite{BosSchmidt-Hieber22}] \label{lem:approximation error for exp}
    For all \(M \geq 2\) and \(\beta > 0\) there exists a neural network \(G \in \mathcal{F}_{\textup{id}}(L, \bm{m}, s)\), with
    \begin{align*}
        &\text{\rm (i) }\; L = \lfloor 40 \pr{\beta + 2}^2 \log_2(M) \rfloor, \\
        &\text{\rm (ii) }\; \bm{m} = (1, \lfloor 48\lceil \beta \rceil^3 2^\beta M^{1 / \beta} \rfloor, \dots , \lfloor 48\lceil \beta \rceil^3 2^\beta M^{1 / \beta} \rfloor, 1), \\
        &\text{\rm (iii) }\; s \leq 4284(\beta + 2)^5 2^\beta M^{1 / \beta} \log_2(M)
    \end{align*}
    such that for any \(x \in [0, 1]\),
    \begin{equation*}
        \verts{e^{G(x)} - x} \leq \frac{4}{M} \quad and \quad G(x) \geq \log\pr{\frac{4}{M}}.
    \end{equation*}
\end{lemma}

\begin{theorem} \label{thm:approximation error for composition structured functions}
    For every function \(\bm{f} \in \mathcal{G}^\prime\pr{r, \bm{d}, \bm{t}, \boldsymbol{\beta}, Q, K}\) and every integer \(N\) satisfies
    \begin{equation*}
        N \geq 2 \quad \text{and} \quad N^{2\beta} \geq c\left(\max_{\bm{\alpha} < \beta}\{\lVert \partial^{\bm{\alpha}} f \rVert_\infty\} \lor 1\right)^{4(\lfloor \beta \rfloor + 1)}
    \end{equation*}
    for some sufficiently large constant \(c \geq 1\), there exist neural networks \(H_k \in \mathcal{F}_\textup{id}(L, \bm{m})\) with
    \begin{align*}
        &\text{\rm (i) }\; L = 8r + 5 + \sum_{i=0}^r \lceil \log_4 (N^{\frac{2\beta_i}{t_i}}) \rceil (\lceil \log_2(\lfloor \beta_i \rfloor \lor t_i + 1) \rceil + 1), \\
        &\text{\rm (ii) }\; \bm{m} = \pr{d_0, m N, \dots , m N, 1},
    \end{align*}
    such that
    \begin{equation*}
        \Verts{H_k - f_k}_\infty \leq C_{Q, r, \boldsymbol{\beta}, \bm{t}, \bm{d}} \max_{i=0, \dots , r} N^{-\frac{2\beta_i^*}{t_i}}, \; \forall k \in [K],
    \end{equation*}
    where \(C_{Q, r, \boldsymbol{\beta}, \bm{t}, \bm{d}}\) is a constant only depending on \(Q, r, \boldsymbol{\beta}, \bm{t}, \bm{d}\) and
    \begin{equation*}
        m \coloneqq \max_{i = 0, \dots , r}d_{i+1} 2^{t_i + 6} \binom{t_i + \lfloor \beta_i \rfloor}{t_i} t_i^2 (\lfloor \beta_i \rfloor + 1).
    \end{equation*}
\end{theorem}
\begin{proof}
    Without loss of generality, we can assume \(Q \geq 1\).
    As in the first part of Lemma \ref{lem:upper bound for composite function}, we can rewrite
    \begin{equation*}
        f_k = \bm{g}_r^{(k)} \circ \cdots \circ \bm{g}_0^{(k)} = \bm{h}_r^{(k)} \circ \cdots \circ \bm{h}_0^{(k)}.
    \end{equation*}
    For each \(h_{ij}^{(k)}\), applying Lemma A.8 in the main paper, there exists a neural network
    \begin{equation*}
        \Tilde{h}_{ij}^{(k)} \in \mathcal{F}_\text{id}\pr{L_i, \pr{t_i, m_i, \dots , m_i, 1}}
    \end{equation*}
    such that
    \begin{equation*}
        \lVert \Tilde{h}_{ij}^{(k)} - h_{ij}^{(k)} \rVert_\infty \leq C \left(\max_{\bm{\alpha} < \beta}\{\lVert \partial^{\bm{\alpha}} h_{ij}^{(k)} \rVert_\infty\} \lor 1\right)^{4(\lfloor \beta_i \rfloor + 1)} N^{-\frac{2\beta_i}{t_i}}.
    \end{equation*}
    Here, 
    \begin{align*}
        &L_i \geq 5 + \lceil \log_4 (N^{\frac{2\beta_i}{t_i}}) \rceil (\lceil \log_2(\lfloor \beta_i \rfloor \lor t_i + 1) \rceil + 1), \\
        &m_i \geq 2^{t_i + 6} \binom{t_i + \lfloor \beta_i \rfloor}{t_i} t_i^2 (\lfloor \beta_i \rfloor + 1) N.
    \end{align*}

    To clip the output of \(\Tilde{h}_{ij}^{(k)}\) to \([0, 1]\), for \(i < r\), we add two layers and apply the transformation \(1 - (1 - x)_+\).
    Denote the transformed network as
    \begin{equation*}
        h_{ij}^{*(k)} \in \mathcal{F}_{\text{id}}\pr{L_i + 2, \pr{t_i, m_i, \dots , m_i, 1, 1, 1}},
    \end{equation*}
    then, it follows that \(\sigma\pr{h_{ij}^{*(k)}} = \pr{\Tilde{h}_{ij}^{(k)}(x) \lor 0} \land 1\).
    Since \(h_{ij}^{(k)}(\bm{x}) \in [0, 1]\), we have 
    \begin{equation*}
        \Verts{\sigma\pr{h_{ij}^{*(k)}} - h_{ij}^{(k)}}_\infty \leq \Verts{\Tilde{h}_{ij}^{(k)} - h_{ij}^{(k)}}_\infty.
    \end{equation*}

    Due to the parallerization property, \(\bm{h}_i^{*(k)} = (h_{ij}^{*(k)})_{j=1, \dots , d_{i+1}}\) belongs to 
    \begin{equation*}
        \mathcal{F}_\text{id}(L_i^\prime + 2, \pr{d_i, d_{i+1}m_i, \dots , d_{i+1}m_i, d_{i+1}}).
    \end{equation*}
    By composing these networks, we construct \(H_k = \Tilde{\bm{h}}_{r}^{(k)} \circ \sigma(\bm{h}_{r-1}^{*(k)}) \circ \cdots \circ \sigma(\bm{h}_0^{*(k)})\).
    From the composition property, \(H_k\) is realized over \(\mathcal{F}_\text{id}\pr{E, \pr{d_0, m, \dots , m, 1}}\) with \(E = \sum_{i=0}^{r-1} (L_i + 2) + r - 1 + L_r + 1 = 3r + \sum_{i=0}^r L_i\) and \(m \coloneqq \max_{i = 0, \dots , r}d_{i+1}m_i\). Therefore, \(H_k\) is a neural network that meets the conditions of the theorem. Combining the inequalities shown earlier with Lemma \ref{lem:upper bound for composite function}, we have
    \begin{align*}
        \Verts{H_k - f_k}_\infty &= \lVert \Tilde{h}_{r1} \circ \sigma(h_{r-1}^{*(k)}) \circ \cdots \circ \sigma(h_0^{*(k)}) - h_r^{(k)} \circ \cdots \circ h_0^{(k)} \rVert_\infty \\
        &\leq Q_r \prod_{l=0}^{r-1}(2Q_l)^{\beta_{l+1}}\sum_{i=0}^r \Verts{\Tilde{h}_i^{(k)} - h_i^{(k)}}_\infty^{\prod_{l=i+1}^r \beta_l \land 1} \\
        &\leq Q_r \prod_{l=0}^{r-1} \pr{2Q_l}^{\beta_{l+1}}\left\{\sum_{i=0}^r\pr{C \left(\max_{\bm{\alpha} < \beta}\{\lVert \partial^{\bm{\alpha}} h_{ij}^{(k)} \rVert_\infty\} \lor 1\right)^{4(\lfloor \beta_i \rfloor + 1)} N^{-\frac{2\beta_i}{t_i}}}^{\prod_{l=i+1}^r \beta_l \land 1}\right\} \\
        &\leq C_{Q, r, \boldsymbol{\beta}, \bm{t}, \bm{d}}\max_{i = 0, \dots , r} N^{-\frac{2\beta_i^*}{t_i}}.
    \end{align*}
\end{proof}

\begin{lemma}\label{lem:approximation error for composition structured probability}
    For any function \(\bm{f} \in \mathcal{G}^\prime\pr{r, \bm{d}, \bm{t}, \boldsymbol{\beta}, Q, K}\) and any $N \in \Nb$ such that
    \begin{equation*}
        N \geq 2,\; N^{2\beta} \geq c\left(\max_{\bm{\alpha} < \beta}\{\lVert \partial^{\bm{\alpha}} f \rVert_\infty\} \lor 1\right)^{4(\lfloor \beta \rfloor + 1)} \quad \text{and} \quad N \geq \max_{i = 0, \dots , r} K^{t_i / (2\beta_i^*)}(4 + C_{Q, r, \boldsymbol{\beta}, \bm{t}})^{t_i / (2\beta_i^*)},
    \end{equation*}
    there exists a neural network \(\Tilde{\bm{p}} \in \mathcal{F}_{\boldsymbol{\Psi}}(L, \bm{m}, s)\) with
        \begin{align*}
    \text{\rm (i) }\;
        &L = \lfloor 40(2\beta_{i^*} / t_{i^*} + 2)^2 \min_{i=0, \dots , r}\log_2 N^{2\beta_i^* / t_i} \rfloor + 8r + 5\\
        &\quad + \sum_{i=0}^r \lceil \log_4 (N^{2\beta_i / t_i}) \rceil (\lceil \log_2(\lfloor \beta_i \rfloor \lor t_i + 1) \rceil + 1),\\
    \text{\rm (ii) }\;
        &\bm{m} = \pr{d_0, mKN, \dots , mKN, K},
   \end{align*}
    where 
    \begin{equation*}
        i^* \coloneqq \argmax_{i = 0, \dots , r} N^{-2\beta_i^* / t_i} \quad \text{and} \quad m \coloneqq \max_{i = 0, \dots , r}d_{i+1} 2^{2 \beta_i + t_i + 6} \binom{t_i + \lfloor \beta_i \rfloor}{t_i} t_i^2 (\lceil \beta_i \rceil^3 + 1),
    \end{equation*}
    such that,
    \begin{equation*}
        \Verts{\Tilde{\bm{p}} - \bm{f}}_\infty \leq 2K(4 + C_{Q, r, \boldsymbol{\beta}, \bm{t}, \bm{d}})\max_{i=0, \dots , r}N^{-\frac{2\beta_i^*}{t_i}},
    \end{equation*}
    and
    \begin{equation*}
        \Tilde{p}_k(\bm{x}) \geq \max_{i=0, \dots , r}N^{-\frac{2\beta_i^*}{t_i}},\; \forall k \in [K],\; \forall \bm{x} \in [0, 1]^{d_0}.
    \end{equation*}
    Here, constant \(C_{Q, r, \boldsymbol{\beta}, \bm{t}, \bm{d}}\) depending only on \(Q, r, \boldsymbol{\beta}, \bm{t}\), and $\bm{d}$.
\end{lemma}
\begin{proof}
    Combining Lemma \ref{lem:approximation error for exp} with \(\beta = 2\beta_{i^*}^* / t_{i^*}\) where \(i^* \coloneqq \argmax_{i = 0, \dots , r} N^{-2\beta_i^* / t_i}\) and Theorem \ref{thm:approximation error for composition structured functions}, we compose the neural network created by them as \(\bm{G} = \pr{G(H_1), \dots , G(H_K)}\). For each \(k = 1, \dots , K\), we have
    \begin{equation*}
        \Verts{e^{G(H_k)} - f_k}_\infty = \Verts{e^{G(H_k)} - H_k}_\infty + \Verts{H_k - f_k}_\infty \leq (4 + C_{Q, r, \boldsymbol{\beta}, \bm{t}, \bm{d}})\max_{i=0, \dots , r}N^{-\frac{2\beta_i^*}{t_i}}. 
    \end{equation*}
    Here, since \(N \geq K^{t_i / (2\beta_i^*)}(4 + C_{Q, r, \boldsymbol{\beta}, \bm{t}, \bm{d}})^{t_i / (2\beta_i^*)} \geq 4^{t_i / (2\beta_i^*)}\), we use \(1 / M = \max_{i = 0, \dots , r}N^{-2\beta_i^* / t_i}\) in Lemma \ref{lem:approximation error for exp}.
    Now, we define the vector-valued function \(\Tilde{\bm{p}}\) element-wise as
    \begin{equation*}
        \Tilde{p}_k(\bm{x}) = \frac{e^{G(H_k(\bm{x}))}}{\sum_{j=1}^K e^{G(H_j(\bm{x}))}},\quad k = 1, \dots , K.
    \end{equation*}
    Applying the composition, depth synchronization, and parallelization property, we establish that \(\Tilde{\bm{p}} \in \mathcal{F}_{\boldsymbol{\Psi}}(L, \bm{m}, s)\). Furthermore, using the triangle inequality and exploiting the fact that \(\bm{f}\) is a probability vector, we obtain
    \begin{align}
        \Verts{\Tilde{p}_k - f_k}_\infty &\leq \Verts{e^{G(H_k)}\pr{\frac{1}{\sum_{j=1}^K e^{G(H_j)}} - 1}}_\infty + \Verts{e^{G(H_k)} - f_k}_\infty \nonumber \\
        &= \Verts{e^{G(H_k)}\pr{\frac{\sum_{l=1}^K f_l}{\sum_{j=1}^K e^{G(H_j)}} - \frac{\sum_{l=1}^K e^{G(H_l)}}{\sum_{j=1}^K e^{G(H_j)}}}}_\infty + \Verts{e^{G(H_k)} - f_k}_\infty \nonumber \\
        &\leq \pr{\sum_{l=1}^K \Verts{f_l - e^{G(H_l)}}_\infty}\Verts{\frac{e^{G(H_k)}}{\sum_{j=1}^K e^{G(H_j)}}}_\infty + \Verts{e^{G(H_k)} - f_k}_\infty \nonumber \\
        &\leq (K+1)(4 + C_{Q, r, \boldsymbol{\beta}, \bm{t}})\max_{i=0, \dots , r}N^{-\frac{2\beta_i^*}{t_i}} \leq 2K(4 + C_{Q, r, \boldsymbol{\beta}, \bm{t}})\max_{i=0, \dots , r}N^{-\frac{2\beta_i^*}{t_i}}. \nonumber
    \end{align}
    Moreover, using the second inequality of Lemma \ref{lem:approximation error for exp} and the first bound of the lemma, we have
    \begin{align*}
        \Tilde{p}_k(\bm{x}) &\geq \frac{4\max_{i=0, \dots , r}N^{-2\beta_i^* / t_i}}{\sum_{j=1}^K e^{G(H_j)}} \geq \frac{4\max_{i=0, \dots , r}N^{-2\beta_i^* / t_i}}{1 + K(4 + C_{Q, r, \boldsymbol{\beta}, \bm{t}})\max_{i=0, \dots , r}N^{-2\beta_i^* / t_i}} \\
        &= \frac{4}{1 / \max_{i=0, \dots , r}N^{-2\beta_i^* / t_i} + K(4 + C_{Q, r, \boldsymbol{\beta}, \bm{t}})} \geq \max_{i=0, \dots , r} N^{-\frac{2\beta_i^*}{t_i}}.
    \end{align*}
\end{proof}

\begin{proof}[Proof of Theorem \ref{thm:convergence rates with non-sparse DNN}]
    Set \(\mathcal{F}_n = \mathcal{F}_{\boldsymbol{\Psi}}\pr{L, \bm{m}, B}\). Also, for a sufficiently small constant \(c^\prime\), define
    \begin{equation*}
        N = \left\lceil c^\prime \max_{i=0, \dots , r} K^{\frac{t_i}{2(\beta_i^* + t_i)}}n^{\frac{t_i}{2(\beta_i^* + t_i)}} \right\rceil, \;
        \phi_n^\prime \coloneqq \max_{i = 0, \dots , r} K^{\frac{2\beta_i^* + 3t_i}{(1 + \alpha)\beta_i^* + t_i}}n^{-\frac{\beta_i^*}{\beta_i^* + t_i}},
    \end{equation*}
    where \(\phi_n^\prime\) represents the convergence rate with the number of classes \(K\).
    Since the Hellinger distance is bounded, the convergence rate for \(\phi_n^\prime \geq 1\) is evident. Thus, consider the case \(\phi_n^\prime \leq 1\). From \(\phi_n^\prime \leq 1\), we have \(K^{t_i / 2\beta_i^*} \leq n^{t_i / 2(2\beta_i^* + 3t_i)} \leq n^{t_i / 2(\beta_i^* + t_i)}\) for all \(i = 0, \dots , r\).
    Thus, for sufficiently large \(n\), the assumption of Lemma \ref{lem:approximation error for composition structured probability} holds. Therefore, take \(\Tilde{\bm{p}}_n\) as \(\Tilde{\bm{p}}\) in Lemma \ref{lem:approximation error for composition structured probability}.

    For any \(\delta > 0\), \(\bar{\mathcal{F}}_n^{1/2}(\delta, \Tilde{\bm{p}}_n) \subset \bar{\mathcal{F}}_n^{1/2}(\Tilde{\bm{p}}_n) \coloneqq \bar{\mathcal{F}}_n^{1/2}(\infty, \Tilde{\bm{p}}_n)\), thus for any \(u>0\),
    \begin{equation*}
        N_{2, B}\pr{u, \bar{\mathcal{F}}_n^{1/2}\pr{\delta, \Tilde{\bm{p}}_n}, \mu} \leq N_{2, B}\pr{u, \bar{\mathcal{F}}_n^{1/2}\pr{\Tilde{\bm{p}}_n}, \mu}.
    \end{equation*}

    Next, we show that for any \(u>0\), \(N_{2, B}(u, \bar{\mathcal{F}}_n^{1/2}(\Tilde{\bm{p}}_n), \mu) \leq N_{2, B}(2\sqrt{2}\max_{i=0, \dots , r}N^{-\beta_i^* / t_i}u, \mathcal{F}_n, \mu).\)
    Let \(\{\pr{\bm{p}_i^U, \bm{p}_i^L}\}_{i=1}^M\) be the smallest \(2\sqrt{2}\max_{i=0, \dots , r}N^{-\beta_i^* / t_i}u\)-bracketing set of \(\mathcal{F}\) for \(L^2(\mu)\) norm, i.e., \(M = N_{2, B}(2\sqrt{2}\max_{i=0, \dots , r}N^{-\beta_i^* / t_i}u, \mathcal{F}_n, \mu)\) and for any \(\bm{p} \in \mathcal{F}_n\), there exists \(j \in [N]\) such that \(\bm{p}_j^L \leq \bm{p} \leq \bm{p}_j^U\) and \(\int (\bm{p}_j^U - \bm{p}_j^L)^2 d\mu \leq 8\max_{i=0, \dots , r}N^{-\beta_i^* / t_i}u^2\). Then, for any \(\bm{p} \in \mathcal{F}_n\), there exists \(j \in [N]\) such that \(\sqrt{(\bm{p}_j^L + \Tilde{\bm{p}}_n)/2} \leq \sqrt{(\bm{p} + \Tilde{\bm{p}}_n)/2} \leq \sqrt{(\bm{p}_j^U + \Tilde{\bm{p}}_n)/2}\), and
    \begin{align*}
        \int \left(\sqrt{\frac{\bm{p}_j^U + \Tilde{\bm{p}}_n}{2}} - \sqrt{\frac{\bm{p}_j^L + \Tilde{\bm{p}}_n}{2}}\right)^2 d\mu
        \leq \frac{1}{8\max_{i=0, \dots , r}N^{-2\beta_i^* / t_i}} \int_\mathcal{X} \sum_{k=1}^K \pr{p_{j, k}^U - p_{j, k}^L}^2 dP_{\bm{X}}(\bm{x})
        \leq u^2.
    \end{align*}
    Thus, the statement is proven.

    From the above, for any \(u>0\), we obtain
    \[
    N_{2, B}(u, \bar{\mathcal{F}}_n^{1/2}\pr{\delta, \Tilde{\bm{p}}_n}, \mu) \leq N_{2, B}\pr{2\sqrt{2}\max_{i=0, \dots , r}N^{-\beta_i^* / t_i}u, \mathcal{F}_n, \mu}.
    \]
    Additionally, for any \(\bm{p}_1, \bm{p}_2 \in \mathcal{F}_{\boldsymbol{\Psi}}(L, \bm{m}, B)\), there exist \(\bm{f}_1, \bm{f}_2 \in \mathcal{F}_{\mathrm{id}}(L, \bm{m}, B)\) such that
    \begin{align*}
        &\Verts{\bm{p}_1 - \bm{p}_2}_\infty = \Verts{\boldsymbol{\Psi}\pr{\bm{f}_1} - \boldsymbol{\Psi}\pr{\bm{f}_2}}_\infty
        = \Verts{\max_{k = 1, \dots , K} \verts{\frac{e^{f_{1, k}(\bm{x})}}{\sum_{l=1, \dots , K} e^{f_{1, l}(\bm{x})}} - \frac{e^{f_{2, k}(\bm{x})}}{\sum_{l=1, \dots , K} e^{f_{2, l}(\bm{x})}}}}_{\infty} \\
        &\leq \Verts{\max_{k = 1, \dots , K} \frac{\sum_{j \neq k} (e^{f_{1, k}(\bm{x}) + f_{2, j}(\bm{x})} + e^{f_{1, j}(\bm{x}) + f_{2, k}(\bm{x})})\verts{f_{1, k}(\bm{x}) + f_{2, j}(\bm{x}) - f_{1, j}(\bm{x}) - f_{2, k}(\bm{x})}}{\pr{\sum_{l=1, \dots , K} e^{f_{1, l}(\bm{x})}}\pr{\sum_{m=1, \dots , K} e^{f_{2, m}(\bm{x})}}}}_{\infty} \\
        &\leq \Verts{\max_{k = 1, \dots , K} \pr{(K-1)\verts{f_{1, k}(\bm{x}) - f_{2, k}(\bm{x})} + \sum_{j \neq k}\verts{f_{1, j}(\bm{x}) - f_{2, j}(\bm{x})}}}_{\infty} \\
        &\leq \Verts{2(K - 1)\max_{k = 1, \dots , K} \verts{f_{1, k}(\bm{x}) - f_{2, k}(\bm{x})}}_{\infty} = 2(K - 1)\Verts{\bm{f}_1 - \bm{f}_2}_\infty.
    \end{align*}
    Here, for the first inequality, we use the fact that  \(\verts{e^x - e^y} \leq (e^x + e^y)\verts{x - y}\) for all $x,y \in \mathbb{R}$. Therefore, combining Lemmas \ref{lem:bracketing and covering with sup norm} and \ref{lem:covering number of DNN}, we have
    \begin{align*}
        \log \left(N_{2, B}\left(2\sqrt{2}\max_{i=0, \dots , r}N^{-\beta_i^* / t_i}u, \mathcal{F}_n, \mu\right)\right)
        &\leq \log \left(N\left(\frac{\sqrt{2}\max_{i=0, \dots , r}N^{-\beta_i^* / t_i}}{\sqrt{K}}u, \mathcal{F}_n, \Verts{\; \cdot \;}_\infty\right)\right) \\
        &\leq \log\pr{N\pr{\frac{\max_{i=0, \dots , r}N^{-\beta_i^* / t_i}u}{(K - 1)\sqrt{2K}}, \mathcal{F}_{\mathrm{id}}(L, \bm{m}, B), \Verts{\; \cdot \;}_\infty}} \\
        &\leq 2sL\log\pr{\frac{\sqrt{2}(K-1)\sqrt{K}(B \lor 1)(\lvert \bm{m} \rvert_\infty + 1)L}{\max_{i=0, \dots , r}N^{-\beta_i^* / t_i}}u^{-1}}.
    \end{align*}
    Here, \(s\) is the total number of parameters of DNN.
    Therefore, let \(A \coloneqq \sqrt{2}(K-1)\sqrt{K}(B \lor 1)(\lvert \bm{m} \rvert_\infty + 1)L / \max_{i=0, \dots , r}N^{-\beta_i^* / t_i}\), we have
    \begin{align*}
        &\int_{\delta^2/(2^{13}c_0)}^\delta \sqrt{\log N_{2, B}\pr{u, \Bar{\mathcal{F}}_n^{1/2}(\Tilde{\bm{p}}_n, \delta), \mu}} \; du
        \leq \sqrt{2sL}\int_0^\delta\sqrt{\log\left(\frac{A}{u}\right)} du \\
        &= \sqrt{2sL}\int_{\sqrt{\log(A/\delta)}}^\infty 2Ax^2e^{-x^2} dx \quad \left(\sqrt{\log(A/u)} = x\right) \\
        &= A\sqrt{2sL}\int_{\sqrt{\log(A/\delta)}}^\infty \left(-e^{-x^2}\right)^\prime x dx 
        = A\sqrt{2sL}\left\{\sqrt{\log(A/\delta)}\frac{\delta}{A} + \sqrt{\pi}\frac{1}{\sqrt{\pi}}\int_{\sqrt{\log(A/\delta)}}^\infty e^{-x^2} dx\right\} \\
        &= A\sqrt{2sL}\left\{\sqrt{\log(A/\delta)}\frac{\delta}{A} + \frac{\sqrt{\pi}}{2}\mathbb{P}\left(\lvert Z \rvert \geq \sqrt{2\log(A/\delta)}\frac{1}{\sqrt{2}}\right)\right\} \quad (Z \sim N(0, 1 / 2)) \\
        &\leq A\sqrt{2sL}\left\{\sqrt{\log(A/\delta)}\frac{\delta}{A} + \frac{\sqrt{\pi}}{2}\frac{2}{\exp\left(\log(A/\delta)\right)}\right\} \quad (\because \text{Gaussian tail bound}) \\
        &= \delta\sqrt{2sL}\left(\sqrt{\log(A/\delta)} + \sqrt{\pi}\right).
    \end{align*}
    Using the triangle inequality and the bound of \ref{lem:approximation error for composition structured probability}, for any \(k \in [K]\) and any \(\bm{x} \in \mathcal{X}\), we have
    \begin{equation*}
        \eta_k(\bm{x}) - \Tilde{p}_{n, k}(\bm{x}) \leq K(4 + C_{Q, r, \boldsymbol{\beta}, \bm{t}}) \max_{i=0, \dots , r}N^{-\frac{2\beta_i^*}{t_i}}\Tilde{p}_{n, k}(\bm{x}) + (4 + C_{Q, r, \boldsymbol{\beta}, \bm{t}, \bm{d}}) \max_{i=0, \dots , r}N^{-\frac{2\beta_i^*}{t_i}}.
    \end{equation*}
    Since \(\Tilde{p}_{n, k}(\bm{x}) \geq \max_{i=0, \dots , r} N^{-2\beta_i^* / t_i}\), dividing both sides by \(\Tilde{p}_{n, k}(\bm{x})\) yields
    \begin{align*}
        \frac{\eta_k(\bm{x})}{\Tilde{p}_{n, k}(\bm{x})} &\leq 1 + K(4 + C_{Q, r, \boldsymbol{\beta}, \bm{t}, \bm{d}}) \max_{i=0, \dots , r}N^{-\frac{2\beta_i^*}{t_i}} + 4 + C_{Q, r, \boldsymbol{\beta}, \bm{t}}
        \leq 6 + C_{Q, r, \boldsymbol{\beta}, \bm{t}, \bm{d}} \leq C_{Q, r, \boldsymbol{\beta}, \bm{t}, \bm{d}}.
    \end{align*}
    Here, the second inequality uses \(N \geq \max_{i = 0, \dots , r}K^{t_i / (2\beta_i^*)}(4 + C_{Q, r, \boldsymbol{\beta}, \bm{t}, \bm{d}})^{t_i / (2\beta_i^*)}\), and the third inequality involves an appropriate adjustment of the constant \(C_{Q, r, \boldsymbol{\beta}, \bm{t}, \bm{d}}\). Therefore, Assumption \ref{assumption} is satisfied, allowing the application of Corollary \ref{cor:expectation oracle inequality}. Setting \(\Psi(\delta) = \delta\sqrt{2sL}(\sqrt{\log(A / \delta)} + \sqrt{\pi})\) and \(c = 1\) in Corollary \ref{cor:expectation oracle inequality}, and choosing \(\delta_n = \sqrt{2sL}(\sqrt{\log(\sqrt{n}A)} + \sqrt{2\pi}) / \sqrt{n}\) to satisfy \eqref{eq:critical inequality}, we obtain
    \begin{equation}
        \mathbb{E}\pr{R\pr{\hat{\bm{p}}_n, \boldsymbol{\eta}}} \leq 514(1 + c_0^2)\pr{\frac{(2sL)(\sqrt{\log(nA)} + \sqrt{2\pi})^2}{n} + R\pr{\Tilde{\bm{p}}_n, \boldsymbol{\eta}}} + \frac{1}{n}. \label{eq:variance bound}
    \end{equation}
    From the proof of Lemma \ref{lem:approximation error for hölder} (see \cite{KohlerLanger21} for details), it follows that the magnitude of the parameters of \(\tilde{\bm{p}}_n\) has order \(\max_{i = 0, \dots , r}N^{(2\beta_i + 2) / t_i}\).
    From this, the choice of \(N\), and depth synchronization property, \(\Tilde{\bm{p}}_n\) belongs to \(\mathcal{F}_{\boldsymbol{\Psi}}(L, \bm{m}, B)\) with \(L, \bm{m}, B\) in the statement of this theorem.
    Further, the bound of Lemma \ref{lem:approximation error for composition structured probability} implies
    \begin{align*}
        R\pr{\Tilde{\bm{p}}_n, \boldsymbol{\eta}} 
        &= \mathbb{E}_{\bm{X}}\sqbr{H^2\pr{\Tilde{\bm{p}}_n, \boldsymbol{\eta}}} 
        = \mathbb{E}_{\bm{X}}\sqbr{\frac{1}{2}\sum_{k=1}^K \pr{\sqrt{\Tilde{p}_k(\bm{X})} - \sqrt{\eta_k(\bm{X})}}^2}\\
        &\leq \mathbb{E}_{\bm{X}}\sqbr{\frac{1}{2}\sum_{k=1}^K \verts{\Tilde{p}_k(\bm{X}) - \eta_k(\bm{X})}}
        \leq K^2(4 + C_{Q, r, \bm{\beta}, \bm{t}, \bm{d}}) \max_{i = 0, \dots, r} N^{-\frac{2\beta_i^*}{t_i}} \lesssim \phi_n^\prime.
    \end{align*}
    Combining this with \eqref{eq:variance bound} and \(s \asymp N^2\log(N) \asymp \max_{i=0, \dots , r} n^{t_i / (\beta_i^* + t_i)} \log(n)\), Theorem \ref{thm:convergence rates with non-sparse DNN} follows.
\end{proof}

\subsection{Proof of Theorem \ref{thm:minimax}}
The proof of the Theorem \ref{thm:minimax} is obtained with non-trivial modifications of the proof of Theorem 3 in \cite{Schmidt-Hieber20}.
\begin{proof}
   Let \(\gamma\) and \(\Gamma\) be the lower and upper bounds of the density function \(p_{\bm{X}}\) of \(\bm{X}\), respectively. According to Theorem 2.7 of \cite{Tsybakov2008}, if for some \(M > 1\) and \(\kappa > 0\), there exist \(\bm{p}_{(0)}, \dots, \bm{p}_{(M)} \in \mathcal{G}^\prime\pr{r, \bm{d}, \bm{m}, \boldsymbol \beta, Q, K}\) such that
    \begin{align*}
        \text{\rm (i) } R\pr{\bm{p}_{(j)}, \bm{p}_{(k)}} \geq \kappa\phi_n
        \;\text{ for all }\;0 \leq j < k \leq M,
        \;\text{ and }\;
        \text{\rm (ii) } 
        \sum_{j = 1}^M\operatorname{KL}\pr{\bm{P}_{(j)} \parallel \bm{P}_{(0)}} \leq \frac{M\log M}{9},
    \end{align*}
    then, there exists a positive constant \(c = c(\kappa, \gamma)\) such that
    \begin{equation*}
        \inf_{\hat{\bm{p}}_n}\sup_{\boldsymbol{\eta} \in \mathcal{G}^\prime\pr{r, \bm{d}, \bm{m}, \boldsymbol \beta, Q, K}} \mathbb{E}_{\mathcal{D}_n}\left[R\pr{\hat{\bm{p}}_n, \boldsymbol{\eta}}\right] \geq c\phi_n.
    \end{equation*}
    Here, \(\bm{P}_{(j)}\) denotes the distribution of the data when the conditional class probability is \(\bm{p}_{(j)}\). Next, we construct \(\bm{p}_{(0)}, \dots, \bm{p}_{(M)} \in \mathcal{G}^\prime\pr{r, \bm{d}, \bm{m}, \boldsymbol \beta, Q, K}\) that satisfy (\romannumeral 1) and (\romannumeral 2). Let \(i^* \in \arg\min_{i= 0, \dots , r}\beta_i^* / (\beta_i^* + t_i)\). For convenience, we denote \(\beta^* \coloneqq \beta_{i^*}, \beta^{**} \coloneqq \beta_{i^*}^*, t^* \coloneqq t_{i^*}\). Assume a function \(\xi \geq 0\) belongs to \(L^1(\mathbb{R}) \cap L^2(\mathbb{R}) \cap \mathcal{C}^{\beta^*}(\mathbb{R}, 1)\) with support on \([0, 1]\) and \(\lVert S \rVert_\infty = 1\). The function \(\xi\) can be taken as, for example, \(\xi(z) \coloneqq b2^{2\beta^*} z^{\beta^*}(1 - z)^{\beta^*}\mathbbm{1}(z \in [0, 1]),\)
    where \(b > 0\) is constant such that \(\xi \in \mathcal{C}^{\beta^*}(\mathbb{R}, T)\).
    Furthermore, let \(m_n \coloneqq \lceil \rho K^{1 / \beta^{**}} n^{1/(\beta^{**} + t^*)} \rceil, h_n \coloneqq 1 / m_n\), and choose the constant \(\rho \geq 1\) such that
    \begin{equation*}
        n(K-1)h_n^{\beta^{**} + t^*} \leq \frac{1}{144\Gamma \log_2 (e) \pr{\lVert \xi^B \rVert_1^{t^*} + \lVert \xi^B \rVert_2^{2t^*}}},
    \end{equation*}
    which \(B := \prod_{l=i^*+1}^r(\beta_l \land 1)\).
    
    \noindent
    For any \(\bm{u} = (u_1, \dots , u_{t^*}) \in \mathcal{U}_n := \{(u_1, \dots , u_{t^*}) \mid u_i \in \{0, h_n, 2h_n, \dots , (m_n - 1)h_n\}\}\), we define
    \begin{equation*}
        \psi_{\bm{u}}(x_1, \dots , x_{t^*}) := h_n^{\beta^*}\prod_{j=1}^{t^*}\xi\left(\frac{x_j - u_j}{h_n}\right).
    \end{equation*}
    Since \(\xi \in \mathcal{C}^{\beta^*}(\mathbb{R}, T)\), for \(\lvert \boldsymbol  \alpha \rvert < \beta^*\), we have \(\lVert \partial^{\boldsymbol \alpha}\psi_{\bm{u}} \rVert_\infty \leq 1.\)
    Moreover, for \(\lvert \boldsymbol \alpha \rvert = \lfloor \beta^* \rfloor\), using the triangle inequality and \(\xi \in \mathcal{C}^{\beta^*}(\mathbb{R}, 1)\), we also have
    \begin{equation*}
        h_n^{\beta^* - \lfloor \beta^* \rfloor}\frac{\left\lvert \prod_{j=1}^{t^*}\xi\left(\frac{x_j - u_j}{h_n}\right) - \prod_{j=1}^{t^*}\xi\left(\frac{y_j - u_j}{h_n}\right)\right\rvert}{\max_{i = 1, \dots , t^*} \lvert x_i - y_i \rvert^{\beta^*-\lfloor\beta^*\rfloor}} \leq t^*.
    \end{equation*}
    Therefore, \(\psi_{\bm{u}} \in \mathcal{C}^{\beta^*}([0, 1]^{t^*}, (\beta^*)^{t^*}t^*)\). For \(\bm{w} = (w_{\bm{u}}) \in \{0, 1\}^{\lvert \mathcal{U}_n \rvert}\), we define $\phi_{\bm{w}} = \sum_{\bm{u} \in \mathcal{U}_n}w_{\bm{u}}\psi_{\bm{u}}.$
    From the construction, \(\psi_{\bm{u}}\) and \(\psi_{\bm{u}^\prime}\) have disjoint supports for \(\bm{u}, \bm{u}^\prime \in \mathcal{U}_n,\; \bm{u} \neq \bm{u}^\prime\). Consequently, \(\phi_{\bm{w}} \in \mathcal{C}^{\beta^*}([0, 1]^{t^*}, 2(\beta^*)^{t^*}t^*)\). For \(i < i^*\), define \(g_i(\bm{x}) := (x_1, \dots , x_{d_i})^\top\); for \(i = i^*\), define \(g_{i^*, \bm{w}}(\bm{x}) := (\phi_{\bm{w}}(x_1, \dots , x_{t_{i^*}}), 0, \dots , 0)^\top\); for \(i > i^*\), define \(g_i(\bm{x}) := (x_1^{\beta_i \land 1}, 0, \dots , 0)^\top\).
    Since \(t_i \leq \min(d_0, \dots , d_{i-1})\) and each \(\psi_{\bm{u}}\) have disjoin support, we have
    \begin{align*}
    f_{\bm{w}} &:= g_r \circ \cdots \circ g_{i^*+1} \circ g_{i^*, \bm{w}} \circ g_{i^*-1} \circ \cdots \circ g_0(\bm{x}) 
    = \phi_{\bm{w}}(x_1, \dots , x_{t_i^*})^B 
    = \sum_{\bm{u} \in \mathcal{U}_n}w_{\bm{u}}\psi_{\bm{u}}(x_1, \dots , x_{t_i^*})^B.
    \end{align*}
    Then, if we choose $Q$ to be sufficiently large, then $f_{\bm{w}}$ belongs to $\mathcal{G}(r, \bm{d}, \bm{m}, \boldsymbol \beta, Q)$.
    Furthermore, for $W = (\bm{w}_1, \dots , \bm{w}_K) \in \{0, 1\}^{(K-1) \times \lvert \mathcal{U}_n \rvert}$, we define
    \begin{equation*}
    \bm{p}_{W} \coloneqq \left(f_{\bm{w}_1}, \dots , f_{\bm{w}_{K-1}}, 1 - \sum_{k=1}^{K-1}f_{\bm{w}_k}\right).
    \end{equation*}
    From the definition of \(m_n\) and \(\xi\), this $\bm{p}_{W}$ belongs to $\mathcal{G}^\prime(r, \bm{d}, \bm{m}, \boldsymbol \beta, Q, K)$.
    It should be noted that $\beta^*B = \beta^{**}$. For any $\bm{u}$ the following equations hold:
    \begin{align*}
    \lVert \psi_{\bm{u}}^B \rVert_1 
    &= h_n^{\beta^*B}\prod_{j=1}^{t^*} \int \xi\left(\frac{x_j - u_j}{h_n}\right)^B dx_j
    = h_n^{\beta^{**} + t^*}\prod_{j=1}^{t^*}\int \xi(x_j)^B dx_j 
    = h_n^{\beta^* + t^*}\lVert \xi^B \rVert_1^{t^*}\;\text{ and } \\
    \lVert \psi_{\bm{u}}^B \rVert_2 &= h_n^{2\beta^* + t^*}\lVert \xi^B \rVert_2^{2t^*}.
    \end{align*}
    
    Here, $\operatorname{Ham}(W, W^\prime) := \sum_{k=1}^K\sum_{\bm{u} \in \mathcal{U}_n}\mathbbm{1}(w_{k, \bm{u}} \neq w_{k, \bm{u}}^\prime)$ represents the Hamming distance, and considering the disjoint support of $\psi_{\bm{u}}$, for $W \neq W^\prime$, we have
    \begin{align*}
    R(\bm{p}_W, \bm{p}_{W^\prime}) &\geq \int \sum_{k=1}^{K-1} \left(\sqrt{f_{\bm{w}_k}(\bm{x})} - \sqrt{f_{\bm{w}_k^\prime}(\bm{x})}\right)^2 dP_{\bm{X}}(\bm{x}) \\
    &= \int \sum_{k=1}^{K-1} \left(\sqrt{\sum_{\bm{u} \in \mathcal{U}_n}w_{k, \bm{u}}\psi_{\bm{u}}^B(\bm{x})} - \sqrt{\sum_{\bm{u} \in \mathcal{U}_n}w_{k, \bm{u}}^\prime\psi_{\bm{u}}^B(\bm{x})}\right)^2 dP_{\bm{X}}(\bm{x}) \\
    &= \int \sum_{k=1}^{K-1} \sum_{\bm{u} \in \mathcal{U}_n}\left(\sqrt{w_{k, \bm{u}}\psi_{\bm{u}}^B(\bm{x})} - \sqrt{w_{k, \bm{u}}^\prime\psi_{\bm{u}}^B(\bm{x})}\right)^2 dP_{\bm{X}}(\bm{x}) \\
    &= \int \sum_{k=1}^{K-1} \sum_{\bm{u} \in \mathcal{U}_n} \mathbbm{1}(w_{k, \bm{u}} \neq w_{k, \bm{u}}^\prime) \psi_{\bm{u}}^B(\bm{x}) dP_{\bm{X}}(\bm{x}) \\
    &\geq \operatorname{Ham}(W, W^\prime) h_n^{\beta^{**} + t^*}\Verts{\xi}_1^{t^*}\gamma.
    \end{align*}
    
    Here, using the Varshamov-Gilbert bound (Lemma 2.9 in \cite{Tsybakov2008}) and the fact that $m_n^{t^*} = \verts{\mathcal{U}_n}$, there exists a subset $\mathcal{W} \subset \{0, 1\}^{(K - 1) \times m_n^{t^*}}$ that includes the matrix where all elements are one and satisfies $\verts{\mathcal{W}} \geq 2^{m_n^{t^*}(K-1)/8}$. Furthermore, for all $W, W^\prime \in \mathcal{W}$, $\operatorname{Ham}(W, W^\prime) \geq m_n^{t^*}(K-1)/8$ holds.
    
    Therefore, by choosing $\kappa = \Verts{\xi^B}_1^{-t^*}\gamma / 8\rho^{\beta^{**}}K$, we have
    $
    R(\bm{p}_W, \bm{p}_{W^\prime})  
    \geq h_n^{\beta^{**}}\Verts{S^B}_1^{t^*}\gamma/8  \geq \kappa\phi_n.
    $
    When $W = (\bm{1}, \dots , \bm{1})$ where $\bm{1} = (1, \dots , 1)^\top$, the distribution of the data is denoted as $\bm{P}_{(0)}$. For any $W \in \mathcal{W}, W \neq (\bm{1}, \dots , \bm{1})$, we have
    \begin{align*}
    \KL{\bm{P}_W}{\bm{P}_{(0)}} 
    &\leq n\Gamma\int \sum_{k=1}^{K-1} \frac{\left(\sum_{\bm{u} \in \mathcal{U}_n}w_{k, \bm{u}}\psi_{\bm{u}}^B(\bm{x}) - \sum_{\bm{u} \in \mathcal{U}_n}\psi_{\bm{u}}^B(\bm{x})\right)^2}{\sum_{\bm{u} \in \mathcal{U}_n}\psi_{\bm{u}}^B(\bm{x})} d\bm{x} \\
    &\quad+ n\Gamma\int \frac{\left(\left(1 - \sum_{k=1}^{K-1}\sum_{\bm{u} \in \mathcal{U}_n}w_{k, \bm{u}}\psi_{\bm{u}}^B(\bm{x})\right) - \left(1 - \sum_{k=1}^{K-1}\sum_{\bm{u} \in \mathcal{U}_n}\psi_{\bm{u}}^B(\bm{x})\right)\right)^2}{1 - \sum_{k=1}^{K-1}\sum_{\bm{u} \in \mathcal{U}_n}\psi_{\bm{u}}^B(\bm{x})} d\bm{x} \\
    &= (\mathrm{I}) + (\mathrm{II}).
    \end{align*}
    
    For $(\mathrm{I})$, considering the disjoint support of $\psi$, we have
    \begin{align*}
    (\mathrm{I}) &= n\Gamma\int \sum_{k=1}^{K-1} \frac{\left(\sum_{\bm{u} \in \mathcal{U}_n}(w_{k, \bm{u}} - 1)\psi_{\bm{u}}^B(\bm{x})\right)^2}{\sum_{\bm{u} \in \mathcal{U}_n}\psi_{\bm{u}}^B(\bm{x})} d\bm{x} 
    = n\Gamma\int \sum_{k=1}^{K-1}\sum_{\bm{u} \in \mathcal{U}_n}\mathbbm{1}(w_{k, \bm{u}} = 0) \frac{\psi_{\bm{u}}^{2B}(\bm{x})}{\sum_{\bm{u} \in \mathcal{U}_n}\psi_{\bm{u}}^B(\bm{x})} d\bm{x} \\
    &\leq n\Gamma\sum_{k=1}^{K-1}\sum_{\bm{u} \in \mathcal{U}_n}\int \psi_{\bm{u}}^B(\bm{x}) d\bm{x} 
    = n\Gamma(K-1)m_n^{t^*}h_n^{\beta^{**} + t^*}\Verts{\xi^B}_1^{t^*}.
    \end{align*}
    
    For $(\mathrm{II})$, we have
    \begin{align*}
    (\mathrm{II}) &= n\Gamma\int \frac{\left(\sum_{k=1}^{K-1}\sum_{\bm{u} \in \mathcal{U}_n} (1 - w_{k, \bm{u}})\psi_{\bm{u}}^B(\bm{x})\right)^2}{1 - \sum_{k=1}^{K-1}\sum_{\bm{u} \in \mathcal{U}_n}\psi_{\bm{u}}^B(\bm{x})} d\bm{x} 
    \leq n\Gamma\int \frac{\sum_{\bm{u} \in \mathcal{U}_n}\mathbbm{1}(w_{k, \bm{u}} = 0)\left(\sum_{k=1}^{K-1}\psi_{\bm{u}}^B(\bm{x})\right)^2}{1 - \sum_{k=1}^{K-1}\sum_{\bm{u} \in \mathcal{U}_n}\psi_{\bm{u}}^B(\bm{x})} d\bm{x} \\
    &\leq n\Gamma (K-1)^2m_n^{t^*}\int \frac{\psi_{\bm{u}}^{2B}(\bm{x})}{1 - (K-1)h_n^{\beta^{**}}} d\bm{x} 
    = n\Gamma (K-1)^2m_n^{t^*}\int \frac{m_n^{\beta^{**}}\psi_{\bm{u}}^{2B}(\bm{x})}{m_n^{\beta^{**}} - (K-1)} d\bm{x} \\
    &\leq n\Gamma(K - 1)^2m_n^{t^*}h_n^{\beta^{**} + t^*}\Verts{\xi^B}_2^{2t^*}.
    \end{align*}
    
    Here, we used the definition of $m_n$ to show that $m_n^{\beta^{**}} - (K-1) \geq 1$. Therefore, using the definition of $\rho$, we obtain
    \begin{equation*}
    \KL{\bm{P}_W}{\bm{P}_{(0)}} \leq 2n\Gamma(K-1)^2m_n^{t^*}h_n^{\beta^{**} + t^*}\left(\Verts{\xi^B}_1^{t^*} + \Verts{\xi^B}_2^{2t^*}\right) 
    \leq \frac{\log \verts{\mathcal{W}}}{9}.
    \end{equation*}
    Thus, the theorem is proven.
\end{proof}

\section{Other results}

\subsection{Convergence rates for anisotropic Besov space} \label{app: convergence rates for Besov space}
In this section, we derive the convergence rate of the maximum likelihood estimator using deep learning when the true conditional class probability \(\bm{\eta}\) belongs to an anisotropic Besov space.
First, we introduce the definition of the anisotropic Besov space. The notation and terminology follow those in \cite{Suzuki&Nitanda2021}.
In this section, for given \(\bm{\beta} = (\beta_1, \dots , \beta_d) \in \mathbb{R}^d\), we write \(\tilde{\bm{\beta}} \coloneqq (\sum_{j=1}^d 1 / \beta_j)^{-1}\) and \(\bar{\bm{\beta}} \coloneqq \max_{j=1, \dots , d} \beta_j\).

Given a function \( f : \mathbb{R}^d \rightarrow \mathbb{R} \), the \( r \)-th difference of \( f \) in the direction \( h \in \mathbb{R}^d \) is defined as
\begin{equation*}
    \Delta^r_h(f)(x) := \Delta^{r-1}_h(f)(x+h) - \Delta^{r-1}_h(f)(x),
\end{equation*}
with \(\Delta^0_h(f)(x) := f(x), \)
where \( x \in [0, 1]^d \) and \( x+rh \in [0, 1]^d \). If this condition is not satisfied, then we define \( \Delta^r_h(f)(x) = 0 \).
Moreover, for a function \( f \in L^p([0, 1]^d) \) where \( p \in (0, \infty] \), the \( r \)-th modulus of smoothness of \( f \) is defined as
\[
w_{r,p}(f,t) = \sup_{h \in \mathbb{R}^d : |h_i| \leq t_i} \|\Delta^r_h(f)\|_p,
\]
where \( t = (t_1, \dots, t_d) \) and \( t_i > 0 \).
Using this modulus of smoothness, we define the anisotropic Besov space \( B^{\bm{\beta}}_{p,q}([0, 1]^d) \) for \( \bm{\beta} = (\beta_1, \dots, \beta_d) \in \mathbb{R}^{d}_{+} \) as follows.
Here, we let \(\mathbb{R}_+ \coloneqq \{x > 0 \mid x \in \mathbb{R}\}\).

\begin{Def} [Anisotropic Besov space]
    For \( 0 < p, q \leq \infty, \bm{\beta} = (\beta_1, \dots, \beta_d) \in \mathbb{R}^{d}_{+}, r := \max_i \beta_i + 1 \), let the seminorm \( |\cdot|_{B^{\bm{\beta}_{p,q}}} \) be
    \[
    |f|_{B^{\bm{\beta}}_{p,q}} := 
    \begin{cases} 
    \left( \sum_{k=0}^{\infty} \left[ 2^k w_{r,p}\left(f,\left(2^{-k/\beta_1}, \dots, 2^{-k/\beta_d}\right)\right) \right]^q \right)^{1/q} & \text{if } q < \infty, \\
    \sup_{k \geq 0} \, 2^k w_{r,p}\left(f,\left(2^{-k/\beta_1}, \dots, 2^{-k/\beta_d}\right)\right) & \text{if } q = \infty.
    \end{cases}
    \]
    The norm of the anisotropic Besov space \( B^{\bm{\beta}}_{p,q}([0, 1]^d) \) is defined by \(\|f\|_{B^{\bm{\beta}}_{p,q}} := \|f\|_{p} + |f|_{B^\beta_{p,q}},\)
    and
    \(B^{\bm{\beta}}_{p,q}([0, 1]^d) = \{ f \in L^p([0, 1]^d) \mid \|f\|_{B^{\bm{\beta}}_{p,q}} < \infty \}\).
\end{Def}
In this section, the function class underlying the true conditional class probabilities \(\boldsymbol{\eta}\) is given by
\begin{equation}
    \mathcal{G}_{\textrm{Besov}}^\prime(p, q, \bm{\beta}, K) \coloneqq \{ \boldsymbol{\eta} = \pr{\eta_1, \dots , \eta_K}^\top : [0, 1]^d \to \mathcal{S}^K \mid \eta_k \in U(B^{\bm{\beta}}_{p,q}([0, 1]^d)), k \in [K] \},
\end{equation}
where \(U(B^{\bm{\beta}}_{p,q}([0, 1]^d))\) be the unit ball of \(B^{\bm{\beta}}_{p,q}([0, 1]^d)\).
Having defined the anisotropic Besov space, we now show the convergence rate of the maximum likelihood estimator using deep learning. The proof is almost identical to that of Theorem 4.2, utilizing Corollary 3.2 in the main paper and Proposition 2 of \cite{Suzuki&Nitanda2021} for evaluating approximation error.
In this section, we consider sparse DNNs as a model.
Let $\lVert W_j \rVert_0$ and $\lvert \bm{v}_j \rvert_0$ denote the number of non-zero components of $W_j$ and the number of non-zero components of $\bm{v}_j$, respectively. 
A class of $s$-sparse networks is defined as
\begin{equation*}
    \mathcal{F}_{\boldsymbol{\psi}}(L, \bm{m}, B, s) = \left\{f \in \mathcal{F}_{\boldsymbol{\psi}}(L, \bm{m}, B) : \sum_{j=0}^L \left(\lVert W_j \rVert_0 + \lvert \bm{v}_j \rvert_0\right) \leq s\right\}.
\end{equation*}
The results for the case where a non-sparse network is used are left for future work.

Next, we present several lemmas concerning the approximation error.
\begin{lemma} [Proposition 2 of \cite{Suzuki&Nitanda2021} with \(r = \infty\)] \label{lem: approximation error for anisotropic Besov}
    Suppose that \( 0 < p, q \leq \infty \) and \( \bm{\beta} \in \mathbb{R}^{d}_{+} \) satisfy the following condition: \( \tilde{\bm{\beta}} > 1 / p \). Assume that \( l \in \mathbb{N} \) satisfies \( 0 < \bar{\bm{\beta}} < \min(l, l - 1 + 1 / p) \). Let \( \delta = 1 / p \), \( \nu = (\tilde{\bm{\beta}} - \delta) / 2\delta \), and \( W_0 := 6d l(l + 2) + 2d \). Then, for any function \(f \in U(B^{\bm{\beta}}_{p,q}([0, 1]^d))\) and \( N \in \mathbb{N} \), there exists a neural network \(\tilde{f} \in \mathcal{F}_{\textup{id}}(L, \bm{m}, s, B)\), with
    \begin{align*}
        &\text{\rm (i) }\; L = 3 + 2 \log_2\left(\frac{3^{(d \vee l)}}{\epsilon c(d,l)}\right) + 5 \log_2(d \vee l),&
        &\text{\rm (ii) }\; \bm{m} = (1, NW_0, \dots , NW_0, 1), \\
        &\text{\rm (iii) }\; s = ((L - 1)W_0^2 + 1)N,&
        &\text{\rm (iv) }\; B \asymp N^{d(1 + \nu^{-1})(1 / p - \tilde{\bm{\beta}})_+},
    \end{align*}
    such that
    \begin{equation*}
        \lVert \tilde{f} - f \rVert_\infty \leq N^{-\tilde{\bm{\beta}}}.
    \end{equation*}
    Here, \( \epsilon = N^{-\tilde{\bm{\beta}}} \log(N)^{-1} \) and a constant \( c(d, l) \) depending only on \( d \) and \( l \).
\end{lemma}

The proof of Lemma \ref{lem:approximation error for anisotropic Besov probability} is almost the same as that of Lemma A.9 in the main paper, but a detailed proof is provided for completeness.
\begin{lemma}\label{lem:approximation error for anisotropic Besov probability}
    For any function \(\bm{f} \in \mathcal{G}_{\textrm{Besov}}^\prime(p, q, \bm{\beta}, K)\) and any $N \in \Nb$ such that \(N \geq 5^{1 / \tilde{\bm{\beta}}} K^{1 / \tilde{\bm{\beta}}}\), there exists a neural network \(\Tilde{\bm{p}} \in \mathcal{F}_{\boldsymbol{\Psi}}(L, \bm{m}, s, B)\) with
        \begin{align*}
    \text{\rm (i) }\;
        &L = \lfloor 40\tilde{\bm{\beta}}(\tilde{\bm{\beta}} + 2)^2 \min_{i=0, \dots , r}\log_2 (\tilde{\bm{\beta}}) \rfloor 4 + 2 \log_2\left(\frac{3^{(d \vee l)}}{\epsilon c(d,l)}\right) + 5 \log_2(d \vee l),\\
    \text{\rm (ii) }\;
        &\bm{m} = \pr{d, KNW_0^\prime, \dots , KNW_0^\prime, K}, \quad
    \text{\rm (iii) }\;
        s \leq (4284(\tilde{\bm{\beta}} + 2)^5 2^{\tilde{\bm{\beta}}} \log_2(N) + (L^\prime - 1)W_0^2 + 1)KN,\\
    \text{\rm (iv) }\;
        &B \asymp N^{d(1 + \nu^{-1})(1 / p - \tilde{\bm{\beta}})_+},
   \end{align*}
    such that,
    \begin{equation*}
        \Verts{\Tilde{\bm{p}} - \bm{f}}_\infty \leq 5KN^{-\tilde{\bm{\beta}}},
    \end{equation*}
    and
    \begin{equation*}
        \Tilde{p}_k(\bm{x}) \geq \max_{i=0, \dots , r}N^{-\tilde{\bm{\beta}}},\; \forall k \in [K],\; \forall \bm{x} \in [0, 1]^{d}.
    \end{equation*}
    Here, \(W_0^\prime \coloneqq 48dl^2(l + 2)2^l + 2d\), \(L^\prime\) is the same \(L\) as in Lemma \ref{lem: approximation error for anisotropic Besov}, and the other symbols are defined in the same manner as in Lemma \ref{lem: approximation error for anisotropic Besov}.
\end{lemma}
\begin{proof}
    Combining Lemma \ref{lem:approximation error for exp} with \(\beta = \tilde{\bm{\beta}}\) and Theorem \ref{lem: approximation error for anisotropic Besov}, we compose the neural network created by them as \(\bm{G} = (G(\tilde{f}_1), \dots , G(\tilde{f}_K))\). For each \(k = 1, \dots , K\), we have
    \begin{equation*}
        \Verts{e^{G(\tilde{f}_k)} - f_k}_\infty = \Verts{e^{G(\tilde{f}_k)} - \tilde{f}_k}_\infty + \Verts{\tilde{f}_k - f_k}_\infty \leq 5 N^{-\tilde{\bm{\beta}}}.
    \end{equation*}
    Here, since \(N \geq 5^{1 / \tilde{\bm{\beta}}} K^{1 / \tilde{\bm{\beta}}} \geq 5^{1 / \tilde{\bm{\beta}}}\), we use \(M = N^{\tilde{\bm{\beta}}}\) in Lemma \ref{lem:approximation error for exp}.
    Now, we define the vector-valued function \(\Tilde{\bm{p}}\) element-wise as
    \begin{equation*}
        \Tilde{p}_k(\bm{x}) = \frac{e^{G(\tilde{f}_k(\bm{x}))}}{\sum_{j=1}^K e^{G(\tilde{f}_j(\bm{x}))}},\quad k = 1, \dots , K.
    \end{equation*}
    Applying the composition, depth synchronization, and parallelization property, we establish that \(\Tilde{\bm{p}} \in \mathcal{F}_{\boldsymbol{\Psi}}(L, \bm{m}, s, B)\). Furthermore, using the triangle inequality and exploiting the fact that \(\bm{f}\) is a probability vector, we obtain
    \begin{align}
        \Verts{\Tilde{p}_k - f_k}_\infty &\leq \Verts{e^{G(\tilde{f}_k)}\pr{\frac{1}{\sum_{j=1}^K e^{G(\tilde{f}_j)}} - 1}}_\infty + \Verts{e^{G(\tilde{f}_k)} - f_k}_\infty \nonumber \\
        &= \Verts{e^{G(\tilde{f}_k)}\pr{\frac{\sum_{l=1}^K f_l}{\sum_{j=1}^K e^{G(\tilde{f}_j)}} - \frac{\sum_{l=1}^K e^{G(\tilde{f}_l)}}{\sum_{j=1}^K e^{G(\tilde{f}_j)}}}}_\infty + \Verts{e^{G(\tilde{f}_k)} - f_k}_\infty \nonumber \\
        &\leq \pr{\sum_{l=1}^K \Verts{f_l - e^{G(\tilde{f}_l)}}_\infty}\Verts{\frac{e^{G(\tilde{f}_k)}}{\sum_{j=1}^K e^{G(\tilde{f}_j)}}}_\infty + \Verts{e^{G(\tilde{f}_k)} - f_k}_\infty \nonumber \\
        &\leq 5(K+1) N^{-\tilde{\bm{\beta}}}. \label{eq:appriximation error for composition structured probability for Besov}
    \end{align}
    Moreover, using the second inequality of Lemma \ref{lem:approximation error for exp} and the first bound of the lemma, we have
    \begin{align*}
        \Tilde{p}_k(\bm{x}) &\geq \frac{4 N^{-\tilde{\bm{\beta}}}}{\sum_{j=1}^K e^{G(H_j)}} \geq \frac{4N^{-\tilde{\bm{\beta}}}}{1 + 5K N^{-\tilde{\bm{\beta}}}} = \frac{4}{N^{\tilde{\bm{\beta}}} + 5K} \geq N^{-\tilde{\bm{\beta}}}.
    \end{align*}
\end{proof}

The following theorem demonstrates the convergence rate of the maximum likelihood estimator using deep learning when the true conditional probability belongs to \(\mathcal{G}_{\textrm{Besov}}^\prime(p, q, \bm{\beta}, K)\).
The proof of this theorem is almost the same as the proof of Theorem 4.1 in the main paper, but a detailed proof is provided for completeness.

\begin{theorem} \label{thm:convergence rates for Besov}
    Consider the K-class classification model
    \begin{equation*}
        Y_k \mid \bm{X} = \bm{x} \sim \text{Bernoulli}(\eta_k(\bm{x})), \quad \bm{X} \sim P_{\bm{X}}, \quad k = 1, \dots , K, 
    \end{equation*}
    for the true conditional class probabilities $\boldsymbol{\eta}$ in the class $\mathcal{G}_{\textrm{Besov}}^\prime(p, q, \bm{\beta}, K)$. Let $\hat{\bm{p}}_n$ be NPMLE with $\mathcal{F}_n = \mathcal{F}_{\boldsymbol{\Psi}}(L, \bm{m}, B, s)$ satisfying the following conditions:
    \[
    \text{\rm (i) }\; L \asymp \log (n),\quad
    \text{\rm (ii) }\; n\phi_n \asymp \min_{i=1, \dots , L}m_i,\quad
    \text{\rm (iii) }\; s \asymp n\phi_n \log(n),\quad
    \text{\rm (iv)}\; B \asymp n^{\frac{d(1 + \nu^{-1})(1 / p - \tilde{\bm{\beta}})_+}{\tilde{\bm{\beta}} + 1}}.
    \]
    Then, there exists a constant $C$ depending only on $p, \bm{\beta}, K$ such that
    \begin{equation*}
        \mathbb{E}_{\mathcal{D}_n}\sqbr{R(\hat{\bm{p}}_n, \boldsymbol{\eta})} \leq C\phi_n\log(n)^3
    \end{equation*}
    for sufficient large \(n\), where \(\phi_n \coloneqq n^{- 1 / (\tilde{\bm{\beta}} + 1)}\)
\end{theorem}

\begin{proof}[Proof of Theorem \ref{thm:convergence rates for Besov}]
    Set \(\mathcal{F}_n = \mathcal{F}_{\boldsymbol{\Psi}}\pr{L, \bm{m}, B, s}\). Also, for a sufficiently small constant \(c^\prime\), define
    \begin{equation*}
        N = \lceil c^\prime K^{\frac{1}{\tilde{\bm{\beta}} + 1}}n^{\frac{1}{\tilde{\bm{\beta}} + 1}} \rceil, \;
        \phi_n^\prime \coloneqq K^{\frac{2\tilde{\bm{\beta}} + 1}{\tilde{\bm{\beta}} + 1}}n^{-\frac{1}{\tilde{\bm{\beta}} + 1}},
    \end{equation*}
    where \(\phi_n^\prime\) represents the convergence rate with the number of classes \(K\).
    Since the Hellinger distance is bounded, the convergence rate for \(\phi_n^\prime \geq 1\) is evident. Thus, consider the case \(\phi_n^\prime \leq 1\). From \(\phi_n^\prime \leq 1\), we have \(K^{t_i / 2\beta_i^*} \leq n^{(1 + \alpha)t_i / 2(2(1 + \alpha)\beta_i^* + 3t_i)} \leq n^{t_i / 2((1 + \alpha)\beta_i^* + t_i)}\) for all \(i = 0, \dots , r\).
    Thus, for sufficiently large \(n\), the assumption of Lemma A.9 in the main paper holds. Therefore, take \(\Tilde{\bm{p}}_n\) as \(\Tilde{\bm{p}}\) in Lemma A.9 in the main paper.

    For any \(\delta > 0\), \(\bar{\mathcal{F}}_n^{1/2}(\delta, \Tilde{\bm{p}}_n) \subset \bar{\mathcal{F}}_n^{1/2}(\Tilde{\bm{p}}_n) \coloneqq \bar{\mathcal{F}}_n^{1/2}(\infty, \Tilde{\bm{p}}_n)\), thus for any \(u>0\),
    \begin{equation*}
        N_{2, B}\pr{u, \bar{\mathcal{F}}_n^{1/2}\pr{\delta, \Tilde{\bm{p}}_n}, \mu} \leq N_{2, B}\pr{u, \bar{\mathcal{F}}_n^{1/2}\pr{\Tilde{\bm{p}}_n}, \mu}.
    \end{equation*}
    Next, we show that for any \(u>0,\; N_{2, B}(u, \bar{\mathcal{F}}_n^{1/2}\pr{\Tilde{\bm{p}}_n}, \mu) \leq N_{2, B}(2N^{-\tilde{\bm{\beta}} / 2}u, \mathcal{F}_n, \mu).\)
    Let \(\{(\bm{p}_i^U, \bm{p}_i^L)\}_{i=1}^M\) be the smallest \(2N^{-\tilde{\bm{\beta}} / 2}u\)-bracketing set of \(\mathcal{F}_n\) for \(L^2(\mu)\) norm, i.e., \(M = N_{2, B}(2N^{-\tilde{\bm{\beta}} / 2}u, \mathcal{F}_n, \mu)\) and for any \(\bm{p} \in \mathcal{F}_n\), there exists \(j \in [M]\) such that \(0 \leq \bm{p}_j^L \leq \bm{p} \leq \bm{p}_j^U\) and \(\int (\bm{p}_j^U - \bm{p}_j^L)^2 d\mu \leq 4N^{-\tilde{\bm{\beta}}}u^2\). Then, for any \(\bm{p} \in \mathcal{F}_n\), there exists \(j \in [M]\) such that \(\sqrt{(\bm{p}_j^L + \Tilde{\bm{p}}_n)/2} \leq \sqrt{(\bm{p} + \Tilde{\bm{p}}_n)/2} \leq \sqrt{(\bm{p}_j^U + \Tilde{\bm{p}}_n)/2}\), and
    \begin{align*}
        \int \left(\sqrt{\frac{\bm{p}_j^U + \Tilde{\bm{p}}_n}{2}} - \sqrt{\frac{\bm{p}_j^L + \Tilde{\bm{p}}_n}{2}}\right)^2 d\mu
        \leq \frac{1}{4N^{-\tilde{\bm{\beta}}}} \int_\mathcal{X} \sum_{k=1}^K \pr{p_{j, k}^U - p_{j, k}^L}^2 dP_{\bm{X}}(\bm{x})
        \leq u^2.
    \end{align*}
    Thus, the statement is proven.

    From the above, for any \(u>0\), we obtain
    \[
    N_{2, B}(u, \bar{\mathcal{F}}_n^{1/2}\pr{\delta, \Tilde{\bm{p}}_n}, \mu) \leq N_{2, B}(2N^{-\tilde{\bm{\beta}} / 2}u, \mathcal{F}_n, \mu).
    \]

    Additionally, for any \(\bm{p}_1, \bm{p}_2 \in \mathcal{F}_{\boldsymbol{\Psi}}(L, \bm{m}, B, s)\), there exist \(\bm{f}_1, \bm{f}_2 \in \mathcal{F}_{\mathrm{id}}(L, \bm{m}, B, s)\) such that
    \begin{align*}
        &\Verts{\bm{p}_1 - \bm{p}_2}_\infty = \Verts{\boldsymbol{\Psi}\pr{\bm{f}_1} - \boldsymbol{\Psi}\pr{\bm{f}_2}}_\infty
        = \Verts{\max_{k = 1, \dots , K} \verts{\frac{e^{f_{1, k}(\bm{x})}}{\sum_{l=1, \dots , K} e^{f_{1, l}(\bm{x})}} - \frac{e^{f_{2, k}(\bm{x})}}{\sum_{l=1, \dots , K} e^{f_{2, l}(\bm{x})}}}}_{\infty} \\
        &\leq \Verts{\max_{k = 1, \dots , K} \frac{\sum_{j \neq k} (e^{f_{1, k}(\bm{x}) + f_{2, j}(\bm{x})} + e^{f_{1, j}(\bm{x}) + f_{2, k}(\bm{x})})\verts{f_{1, k}(\bm{x}) + f_{2, j}(\bm{x}) - f_{1, j}(\bm{x}) - f_{2, k}(\bm{x})}}{\pr{\sum_{l=1, \dots , K} e^{f_{1, l}(\bm{x})}}\pr{\sum_{m=1, \dots , K} e^{f_{2, m}(\bm{x})}}}}_{\infty} \\
        &\leq \Verts{\max_{k = 1, \dots , K} \pr{(K-1)\verts{f_{1, k}(\bm{x}) - f_{2, k}(\bm{x})} + \sum_{j \neq k}\verts{f_{1, j}(\bm{x}) - f_{2, j}(\bm{x})}}}_{\infty} \\
        &\leq \Verts{2(K - 1)\max_{k = 1, \dots , K} \verts{f_{1, k}(\bm{x}) - f_{2, k}(\bm{x})}}_{\infty} = 2(K - 1)\Verts{\bm{f}_1 - \bm{f}_2}_\infty.
    \end{align*}
    Here, for the first inequality, we use the fact that  \(\verts{e^x - e^y} \leq (e^x + e^y)\verts{x - y}\) for all $x,y \in \mathbb{R}$. Therefore, by combining Lemma A.6 in the main paper and Lemma A.7 in the main paper, we have
    \begin{align*}
        \log (N_{2, B}(2N^{-\tilde{\bm{\beta}} / 2}u, \mathcal{F}_n, \mu))
        &\leq \log \left(N\left(\frac{N^{-\tilde{\bm{\beta}} / 2}}{\sqrt{K}}u, \mathcal{F}_n, \Verts{\; \cdot \;}_\infty\right)\right) \\
        &\leq \log\pr{N\pr{\frac{N^{-\tilde{\bm{\beta}} / 2}u}{2(K - 1)\sqrt{K}}, \mathcal{F}_{\mathrm{id}}(L, \bm{m}, B, s), \Verts{\; \cdot \;}_\infty}} \\
        &\leq 2sL\log\pr{\frac{2(K-1)\sqrt{K}(B \lor 1)(\lvert \bm{m} \rvert_\infty + 1)L}{N^{-\tilde{\bm{\beta}} / 2}}u^{-1}}.
    \end{align*}
    Therefore, let \(A \coloneqq 2(K-1)\sqrt{K}(B \lor 1)(\lvert \bm{m} \rvert_\infty + 1)L / N^{-\tilde{\bm{\beta}} / 2}\), we have
    \begin{align*}
        &\int_{\delta^2/(2^{13}c_0)}^\delta \sqrt{\log N_{2, B}\pr{u, \Bar{\mathcal{F}}_n^{1/2}(\Tilde{\bm{p}}_n, \delta), \mu}} \; du \\
        &\leq \sqrt{2sL}\int_0^\delta\sqrt{\log\left(\frac{A}{u}\right)} du
        = \sqrt{2sL}\int_{\sqrt{\log(A/\delta)}}^\infty 2Ax^2e^{-x^2} dx \quad \left(\sqrt{\log(A/u)} = x\right) \\
        &= A\sqrt{2sL}\int_{\sqrt{\log(A/\delta)}}^\infty \left(-e^{-x^2}\right)^\prime x dx 
        = A\sqrt{2sL}\left\{\sqrt{\log(A/\delta)}\frac{\delta}{A} + \sqrt{\pi}\frac{1}{\sqrt{\pi}}\int_{\sqrt{\log(A/\delta)}}^\infty e^{-x^2} dx\right\} \\
        &= A\sqrt{2sL}\left\{\sqrt{\log(A/\delta)}\frac{\delta}{A} + \frac{\sqrt{\pi}}{2}\mathbb{P}\left(\lvert Z \rvert \geq \sqrt{2\log(A/\delta)}\frac{1}{\sqrt{2}}\right)\right\} \quad (Z \sim N(0, 1 / 2)) \\
        &\leq A\sqrt{2sL}\left\{\sqrt{\log(A/\delta)}\frac{\delta}{A} + \frac{\sqrt{\pi}}{2}\frac{2}{\exp\left(\log(A/\delta)\right)}\right\} \quad (\because \text{Gaussian tail bound}) \\
        &= \delta\sqrt{2sL}\left(\sqrt{\log(A/\delta)} + \sqrt{\pi}\right).
    \end{align*}
    Using the triangle inequality and \eqref{eq:appriximation error for composition structured probability for Besov}, for any \(k \in [K]\) and any \(\bm{x} \in \mathcal{X}\), we have
    \begin{equation*}
        \eta_k(\bm{x}) - \Tilde{p}_{n, k}(\bm{x}) \leq 5(K + 1)N^{-\tilde{\bm{\beta}}}.
    \end{equation*}
    Since \(\Tilde{p}_{n, k}(\bm{x}) \geq N^{-\tilde{\bm{\beta}}}\), dividing both sides by \(\Tilde{p}_{n, k}(\bm{x})\) yields
    \begin{equation*}
        \frac{\eta_k(\bm{x})}{\Tilde{p}_{n, k}(\bm{x})} \leq 5K + 6.
    \end{equation*}
    Therefore, Assumption 2.1 in the main paper is satisfied, allowing the application of Corollary 3.2 in the main paper.
    
    Setting \(\Psi(\delta) = \delta\sqrt{2sL}(\sqrt{\log(A / \delta)} + \sqrt{\pi})\) and \(c = 1\) in Corollary 3.2 in the main paper, and choosing \(\delta_n = \sqrt{2sL}(\sqrt{\log(\sqrt{n}A)} + \sqrt{2\pi}) / \sqrt{n}\) to satisfy (8) in the main paper, we obtain
    \begin{equation}
        \mathbb{E}\pr{R\pr{\hat{\bm{p}}_n, \boldsymbol{\eta}}} \leq 514(1 + c_0^2)\pr{\frac{(2sL)(\sqrt{\log(nA)} + \sqrt{2\pi})^2}{n} + R\pr{\Tilde{\bm{p}}_n, \boldsymbol{\eta}}} + \frac{1}{n}. \label{eq:variance bound for Besov}
    \end{equation}
    From the choice of \(N\) and depth synchronization property, \(\Tilde{\bm{p}}_n\) belongs to \(\mathcal{F}_{\boldsymbol{\Psi}}(L, \bm{m}, s, B)\) with \(L, \bm{m}, s, B\) in the statement of this theorem.
    Further, the bound of Lemma \ref{lem:approximation error for anisotropic Besov probability} implies
    \begin{align*}
        R\pr{\Tilde{\bm{p}}_n, \boldsymbol{\eta}} 
        &= \mathbb{E}_{\bm{X}}\sqbr{H^2\pr{\Tilde{\bm{p}}_n, \boldsymbol{\eta}}} 
        = \mathbb{E}_{\bm{X}}\sqbr{\frac{1}{2}\sum_{k=1}^K \pr{\sqrt{\Tilde{p}_k(\bm{X})} - \sqrt{\eta_k(\bm{X})}}^2}\\
        &\leq \mathbb{E}_{\bm{X}}\sqbr{\frac{1}{2}\sum_{k=1}^K \verts{\Tilde{p}_k(\bm{X}) - \eta_k(\bm{X})}}
        \leq 5K^2 N^{-\tilde{\bm{\beta}}} \lesssim \phi_n^\prime.
    \end{align*}
    Combining this with \eqref{eq:variance bound for Besov} and \(s \asymp N\log(N) \asymp n^{1 / (\tilde{\bm{\beta}} + 1)} \log(n)\), Theorem \ref{thm:convergence rates for Besov} follows.
\end{proof}

\end{appendix}

\section*{Acknowledgments}
This research was supported by 
JST BOOST (JPMJBS2402 to AY), 
JSPS KAKENHI Grant (JP20K19756, JP20H00601, 23K28045, 24K14855 to YT),
the MEXT Project for Seismology Toward Research Innovation with Data of Earthquakes (STAR-E, JPJ010217 to YT), and iGCORE Open Collaborative Research (to YT).
\bibliographystyle{plainnat}  
\bibliography{bibliography}  

\end{document}